\documentclass[preprints,article,accept,moreauthors,pdftex]{Definitions/mdpi}

\firstpage{1}
\pubvolume{xx}
\issuenum{1}
\articlenumber{5}
\pubyear{2020}
\copyrightyear{2020}
\history{Received: date; Accepted: date; Published: date}

\usepackage{amssymb}
\usepackage{amsmath}
\usepackage{amsfonts}
\usepackage{graphicx}
\usepackage{float}
\usepackage{subcaption}
\graphicspath{ {./images/} }

\newtheorem{corollary}[theorem]{Corollary}

\newtheorem{definition}[theorem]{Definition}

\newtheorem{lemma}[theorem]{Lemma}

\newtheorem{proposition}[theorem]{Proposition}

\Title{The Finsler spacetime condition for $(\alpha,\beta)$-metrics and their isometries}

\Author{Nicoleta Voicu$^{1,\ddagger}$,Annam\' aria Friedl-Sz\' asz$^{1,\ddagger}{}$, Elena Popovici-Popescu$^{1,\ddagger}{}^*$ and Christian Pfeifer}

\AuthorNames{Annam\' aria Friedl-Sz\' asz, Elena Popovici-Popescu, Nicoleta Voicu and Christian Pfeifer}

\address{%
$^{1}$ \quad Faculty of Mathematics and Computer Science, Transilvania University, Iuliu Maniu Str. 50, 500091 Brasov, Romania; nico.voicu@unitbv.ro\\
$^{2}$ \quad ZARM, University of Bremen, 28359 Bremen, Germany; christian.pfeifer@zarm.uni-bremen.de\\
}

\corres{Correspondence:nico.voicu@unitbv.ro}

\secondnote{These authors contributed equally to this work.}

\abstract{For the general class of pseudo-Finsler spaces with $(\alpha,\beta)$-metrics, we establish necessary and sufficient conditions such that these admit a Finsler spacetime structure. This means that the
fundamental tensor has Lorentzian signature on a conic subbundle of the tangent bundle and thus the
existence of a cone of future pointing timelike vectors is ensured. The identified $(\alpha,\beta)$-Finsler
spacetimes are candidates for applications in gravitational physics. Moreover, we completely determine
the relation between the isometries of an $(\alpha,\beta)$-metric and the isometries of the underlying pseudo-Riemannian
metric $a$; in particular, we list all  $(\alpha,\beta)$-metrics which admit isometries that are not isometries of $a$.}


\begin{document}

\tableofcontents 






\section{Introduction}
Finsler geometry, which is the geometry of a manifold described by a general geometric length measure for curves, has numerous applications in physics \cite{P}. In the context of gravitational physics, it is the perfect mathematical  framework to describe the gravitational field of a kinetic gas \cite{CANTATA:2021ktz,HPV2020}, it describes the propagation of particles subject to deformed/doubly special relativity symmetries employed in quantum gravity phenomenology \cite{Addazi:2021xuf,Lobo:2020qoa,Amelino-Camelia:2014rga}, and it emerges naturally in the context of theories  based on broken/deformed Lorentz invariance such as for example the standard model extension (SME) or very special relativity (VSR) \cite{Gibbons,Kostelecky}. 

In general, Finsler geometry allows for a huge variety of  structures, which are way vaster than the variety of pseudo-Riemannian structures on manifolds. Therefore, it is important to classify Finsler geometries, in order to identify the best models for specific applications.


Among all Finsler structures, the class of $(\alpha,\beta)$-metrics, obtained by constructing a geometric length measure for curves from a (pseudo)-Riemannian metric $a$ and a 1-form $b$, are the easiest to construct and the most used in practice. Notorious examples include: Bogoslovsky-Kropina (or $m$-Kropina) metrics, which represent the framework for VSR and its generalization,  very general relativity (VGR) \cite{Gibbons, CG, FPa, Fuster-VGR, Bog-Fins} -- also used for dark energy models \cite{Bouali:2023flv} -- and Randers metrics, used, for instance in the description of propagation of light in static spacetimes \cite{Werner:2012rc,math-foundations}, for the motion of an electrically charged particle in an electromagnetic field,  in the study of Finsler gravitational waves \cite{Heefer:2020hra}, or in the SME \cite{Shreck, KRT, S, math-foundations}.

All the above mentioned applications require Finsler metrics of Lorentzian signature. While there exists a rich literature on positive definite Finsler metrics (and in particular, on $(\alpha,\beta)$-ones, \cite{Bacso,SS,Mats2,Mo2,Elgendi_2020,Crampin_2022}), Lorentzian Finsler geometry is by far less understood and investigated. 
In this paper, we study for the first time in full generality two questions about Lorentzian $(\alpha,\beta)$-Finsler structures, which have been just partially tackled in the literature (mostly only for very particular cases): 
\begin{enumerate}
	\item the necessary and sufficient conditions for an $(\alpha,\beta)$-metric to define a Finsler spacetime structure;\vspace{5pt}

	\item determining the isometries of general $(\alpha,\beta)$-metrics.
\end{enumerate}
 
Speaking about the first problem enumerated above, the very definition of a Finsler spacetime is actually still a matter of debate, \cite{Pfeifer:2011tk,Lammerzahl:2012kw,JS,Hasse:2019zqi,Bernal2020, Caponio-Masiello,Caponio-Stancarone, math-foundations, Beem, Asanov}. Yet, in recent years, though the various definitions may differ in minute details, they all converge to the following understanding: at each point of a Finslerian spacetime, there should exist a convex cone with null boundary - interpreted as the cone of future-pointing timelike vectors - on which the Finsler metric tensor must be well defined, smooth (maybe with the exception of one singular direction, \cite{Caponio-Stancarone}) and with Lorentzian signature. 

Starting from this understanding, we determine the conditions for a general $(\alpha,\beta)$-metric, with completely arbitrary 1-form, to be smooth and to have Lorentzian signature inside such a cone. Also, we present concrete examples that are interesting for applications such as Randers and Bogoslovsky-Kropina metrics (extending previous studies \cite{HPV2019,FPa} for the case of non-spacelike 1-forms) as well as Kundt and exponential metrics.

For the second problem, isometries of $(\alpha,\beta)$-metrics, to the best of our knowledge, the only cases when these were known are: Bogoslovky-Kropina deformations of Minkowski metric \cite{Bogoslovsky77, Mo2}, as well as Randers and Kropina metrics \cite{Mats2}. Here, we determine infinitesimal isometries of \textit{general} $(\alpha,\beta)$-metrics.  

The structure of this paper is as follows. Section \ref{sec:prelim} reviews the necessary notions of Finsler spacetimes for our later construction. Section \ref{sec:stcon}  consists in the investigation of the conditions for an $(\alpha,\beta)$-metric to define a spacetime structure, and presents our main theorem, Theorem \ref{thm:sign_g}. A complete classification, using simple conditions, is then given for the most used classes in Section \ref{sec:ex}. In Section \ref{sec:isom}, we determine the infinitesimal isometries of general $(\alpha,\beta)$-metrics. Section \ref{sec:conc} briefly presents our conclusions. In the Appendix \ref{app:detg}, we display the proof of our formula for the determinant of the fundamental tensor of a general $(\alpha,\beta)$-metric and of its inverse. 

\section{Preliminaries}\label{sec:prelim}
We begin by recalling the concept of a Finsler spacetime in this section. 

There are numerous attempts to find a suitable definition of a Finsler spacetime, i.e.\ of pseudo-Finsler geometry with a Finsler metric of Lorentzian signature, where among the first are the one by Beem \cite{Beem} and Asanov \cite{Asanov}. However, it quickly turned out that these  definitions given are too restrictive to cover numerous interesting physical examples, such as m-th root metrics, Randers metrics or m-Kropina metrics. Since then, several definition of Finsler spacetimes have been developed \cite{Pfeifer:2011tk,Lammerzahl:2012kw,JS,Hasse:2019zqi,Bernal2020, Caponio-Masiello,Caponio-Stancarone, math-foundations}, all agreeing that the Finsler metric tensor should be of Lorentzian signature on some (conic) subset of the tangent bundle, but differing in the precise details of where it must be smooth or only continuous. The origin of these fine differences lies in the various applications and examples on which the authors focused when formulating their definitions; thorough discussions of the differences between the distinct approaches to indefinite Finsler spacetime geometry can be found, e.g., in \cite{JS, math-foundations}.

In the following, we will use the  notion of Finsler spacetime as defined by two of us in \cite{math-foundations}, as it is the most permissive one which still allows for well defined curvature-related quantities on the entire future-pointing timelike domain. Yet, as we will point out below, our approach can be applied with a minimal modification to the (even more permissive) definition by Caponio\&co., \cite{Caponio-Stancarone, Caponio-Masiello}.

\bigskip

Prior to introducing the notion of Finsler spacetime, we briefly introduce the manifolds we are working on and the preliminary notions of conic subbundle and pseudo-Finsler structure.

\bigskip

For the whole article, let $M$ be a 4-dimensional connected, orientable smooth manifold, $(TM,\pi
,M),$ its tangent bundle and $\overset{\circ }{TM}=TM\backslash \{0\},$ the
tangent bundle without its zero section. We will denote by $(x^{i})_{i=%
	\overline{0,3}},$ the coordinates of a point $x\in U\subset M$ in a local
chart $\left( U,\varphi \right) $ and by $(x^{i},\dot{x}^{i})$, the
naturally induced local coordinates of points $(x,\dot{x})\in \pi ^{-1}(U)$.
Commas $_{,i}$ will mean partial differentiation with respect to the
coordinates $x^{i}$ and dots $_{\cdot i}$ partial differentiation with to
coordinates $\dot{x}^{i}.$ Also, whenever there is no risk of confusion, we
will omit for simplicity the indices of the coordinates.

\bigskip

A \emph{conic subbundle} of $TM$ is a non-empty open submanifold $\mathcal{Q}%
\subset \overset{\circ }{TM}$, which projects by $\pi $ on the entire
manifold, i.e., $\pi (\mathcal{Q})=M$, and possessing the so-called \textit{%
	conic property:} 
\begin{equation*}
	(x,\dot{x})\in \mathcal{Q~}\Rightarrow ~(x,\lambda \dot{x})\in \mathcal{Q}%
	,~\ \ \forall \lambda >0.
\end{equation*}

Further, a \textit{pseudo-Finsler structure }on $M$, see \cite{Bejancu}, is a smooth function $L:\mathcal{A}\rightarrow \mathbb{R}$ defined on a conic
subbundle $\mathcal{A}\subset ~\overset{\circ }{TM}$, obeying the
following conditions:

\begin{enumerate}
	\item positive 2-homogeneity:\ $L(x,\alpha \dot{x})=\alpha ^{2}L(x,\dot{x}),\,
	\forall \alpha >0,$ $\forall (x,\dot{x})\in \mathcal{A}.$ \vspace{5pt}
	
	\item at any $\left( x,\dot{x}\right) \in \mathcal{A}$ and in one (and then,
	in any) local chart around $(x,\dot{x}),$ the Hessian: 
	\begin{equation}
		g_{ij}=\dfrac{1}{2}\dfrac{\partial ^{2}L}{\partial \dot{x}^{i}\partial \dot{x%
			}^{j}}=\dfrac{1}{2}L_{\cdot i\cdot j}  \label{g_ij}
	\end{equation}%
	is nondegenerate.
\end{enumerate}

We note that, in general, the functions $g_{ij}=g_{ij}(x,\dot{x})$
have a nontrivial $\dot{x}$-dependence; more precisely, they define a mapping%
\begin{equation}
	g:\mathcal{A}\rightarrow T_{2}^{0}M,\left( x,\dot{x}\right) \mapsto g_{(x,%
		\dot{x})}=g_{ij}(x,\dot{x})dx^{i}\otimes dx^{j},  \label{metric_tensor}
\end{equation}%
called the \textit{Finslerian metric tensor. }This is generally, not a
tensor field on $M$ (as it depends on vectors $\dot{x}\in T_{x}M$), but it
plays a largely similar role to the one of the metric tensor in
pseudo-Riemannian geometry. The particular case when $g=g_{(x)} $ only (that is, $L(x,\dot{x}%
)=a_{ij}(x)\dot{x}^{i}\dot{x}^{j}$ is \textit{quadratic} in $\dot{x}$)
corresponds to pseudo-Riemannian geometry, see for example \cite{BCS}.

\bigskip

\begin{definition}[Finsler spacetime, following \cite{math-foundations}]\label{def_Finsler_spacetime}
	A Finsler spacetime is a 4-dimensional pseudo-Finsler space $\left( M,L\right) ,$ $L:\mathcal{A}%
	\rightarrow \mathbb{R},$ (with $M$ - connected and orientable) obeying the additional third
	condition:
\end{definition}

\begin{enumerate}
	\item[3.] There exists a connected conic subbundle $\mathcal{T}\subset 
	\mathcal{A}$ with connected fibers $\mathcal{T}_{x}=\mathcal{T}\cap T_{x}M,$ 
	$x\in M,$\textbf{\ }such that, on each $\mathcal{T}_{x}:$ $L>0,$ $g$ has
	Lorentzian signature $(+,-,-,-)$ and $L$ can be continuously extended as $0$
	to the boundary $\partial \mathcal{T}_{x}.$
\end{enumerate}

Physically, the \textit{Finsler spacetime function} $L$ is interpreted as the interval $ds^{2}=L(x,dx)$.

\bigskip

On a Finsler spacetime there exist the following important subsets of $TM$:
\begin{enumerate}
	\item The conic subbundle $\mathcal{A}$ where $L$ is defined, smooth and
	with nondegenerate Hessian is called the set of \textit{admissible vectors};
	we will typically understand by $\mathcal{A},$ the \textit{maximal} subset
	of $\overset{\circ }{TM}$ with these properties. \vspace{5pt}
	
	\item The conic subbundle $\mathcal{T},$ where the signature of $g$ and the sign of $L$ agree, will be interpreted as the set of \textit{future pointing timelike vectors}.
\end{enumerate}

\textbf{Note.} From the above definition, it follows, that all the fibers $\mathcal{T}_{x}$ of $\mathcal{T}$ are, actually, \textit{convex}
cones, see \cite{math-foundations}.

The definition in \cite{JS} can be recovered by setting $\mathcal{A}=\mathcal{T}$ and imposing that $L$ extends \textit{smoothly} to the boundary $\partial\mathcal{T}$, whereas the one in \cite{Caponio-Masiello, Caponio-Stancarone}, by allowing $L$ to be of class $\mathcal{C}^{1}$ only  along one direction in each cone $\mathcal{T}_{x}$.

\bigskip

In a Finsler spacetime, the arc length of a curve  $\gamma :[a,b]\rightarrow
M,~\ t\mapsto \gamma (t):$%
\begin{equation*}
	l_{\gamma }=\overset{b}{\underset{a}{\int }}\sqrt{\left\vert L\left( \gamma
		(t),\dot{\gamma}(t)\right) \right\vert }dt
\end{equation*}%
is well defined (independent of the parametrization), by virtue of the
2-homogeneity of $L.$ Moreover, for future-pointing timelike curves (defined
by the fact that $\dot{\gamma}(t)=\tfrac{d\gamma (t)}{dt}$ belongs to the cone$\mathcal{T}%
_{\gamma (t)}$, for all $t$ - and interpreted as worldlines of massive
particles), it gives the \textit{proper time} along the respective worldline.

Having clarified our notion of Finsler spacetimes, we can proceed and present general conditions for $(\alpha,\beta)$-Finsler metrics, so that they indeed define Finsler spacetimes.

\bigskip
\section{Spacetime conditions for $\left( \protect\alpha ,\protect\beta \right) $-metrics}\label{sec:stcon}
Finsler metrics of $(\alpha,\beta)$-type can nicely be classified and studied, due to their fairly simple building blocks, which are a pseudo-Riemannian metric  $a$ and a 1-form $b$ on $M$.

More precisely, consider a Lorentzian metric $%
a=a_{ij}dx^{i}\otimes dx^{j}$ (we use the signature convention $\left(+,-,-,-\right) $) and a 1-form $b=b_{i}dx^{i}$ on $M$. Let us denote, in any local
chart: 
\begin{equation}
	B=b(\dot{x})=b_{i}\dot{x}^{i},~\ \ A\left( x,\dot{x}\right) =a_{x}\left( 
	\dot{x},\dot{x}\right) =a_{ij}\dot{x}^{i}\dot{x}^{j};  \label{def_AB}
\end{equation}%
in comparison to the usual notations $\alpha, \beta$ in the literature on positive definite Finsler spaces, this is:

\begin{align}
\alpha = \sqrt{|A|},\,\,  \beta=B.
\end{align}

\noindent Throughout the paper, for simplicity, we will also refer to the Lorentzian metric $a$ as $A$.

An $\left( \alpha ,\beta \right) $\textit{-metric }on $M$ is, by definition,
a pseudo-Finsler structure $L:\mathcal{A}\rightarrow \mathbb{R}$ (where $%
\mathcal{A}\subset \overset{\circ }{TM}$ is a conic subbundle), of the form: 
\begin{equation}
	L=A\Psi (s),~\ \ \ s=\dfrac{B^{2}}{A},  \label{def_alpha_beta_L}
\end{equation}%
where $\Psi =\Psi (s)$ is smooth on the set $s(\mathcal{A})$.

\bigskip

The metric tensor components $g_{ij}=g_{ij}(x,\dot{x})$ are easily found as:
\begin{equation}
	g_{ij}=a_{ij}(\Psi -s\Psi ^{\prime })+\Psi ^{\prime }b_{i}b_{j}+\dfrac{1}{2}%
	A\Psi ^{\prime \prime }s_{\cdot i}s_{\cdot j}  \label{g_ij_Phi}\,,
\end{equation}%
with inverse \eqref{g^ij_Phi} given in  Appendix \ref{app:detg}, see also \cite{Fuster-VGR}.

\bigskip

In the following, we will investigate the conditions such that functions $L$
as in (\ref{def_alpha_beta_L}) define a Finsler spacetime structure on $M.$
In other words, we will investigate the existence of a conic subbundle $%
\mathcal{T\subset A}$ satisfying the requirement of Definition \ref%
{def_Finsler_spacetime}. With this aim, let us make the following assumption:

\vspace{5pt}
\textbf{Assumption. }\textit{At each }$x\in M,$\textit{\ the cone }$\mathcal{%
	T}_{x}$\textit{\ of }$L\ $\textit{and the future pointing timelike cone $%
	\mathcal{T}_{x}^{a}$ of the metric }$ a $\textit{\ have a non-empty intersection.}\\
In other words, subluminal speeds of particles, as measured by a Finslerian observer, are not all superluminal from the point of view of an observer using a (pseudo-)Riemannian arc length.
\bigskip

Fix, in the following, an arbitrary point $x\in M$ and an arbitrary local
chart, with local coordinates $\left( x^{i},\dot{x}^{i}\right) $ in a
neighborhood of $\pi ^{-1}(x).$ In the following, we will omit the point $x$
from the writing of $\left( x,\dot{x}\right) $-dependent quantities, that
is, we will write simply $L=L(\dot{x}),$ $A=A(\dot{x}),$ $g_{\dot{x}}=g_{(x,%
	\dot{x})}$ etc.

We denote by $\left\langle ~~~,~~~\right\rangle $ the scalar product with
respect to $a_{x}$, that is:%
\begin{equation}
	A :=\left\langle \dot{x},\dot{x}\right\rangle ,~\
	\left\langle b,b\right\rangle =a^{ij}b_{i}b_{j}=a_{ij}\tilde{b}^{i}\tilde{b}%
	^{j},~\ \left\langle \tilde{b},\dot{x}\right\rangle =b( \dot{x})
	=B,  \label{AB_scalar_prod}
\end{equation}%
where tildes are used to mark raising/lowering of indices by $a,$ e.g., $%
\tilde{b}^{i}=a^{ij}b_{j}$ (in contrast, we will use no tildes when
raising/lowering indices with $g,$ that is: $b^{i}=g^{ij}b_{j}$ etc.).
\bigskip

\textbf{Remarks.}
\begin{enumerate}
	\item Using the above assumption $\mathcal{T}_{x}\cap \mathcal{T}%
	_{x}^{a}\not=\varnothing $ and $L_{|\mathcal{T}_{x}}>0$, it follows that:
	\begin{equation}
	\forall \dot{x}\in \mathcal{T}_{x}:~A>0,~\ \Psi >0;
	\label{signs_T_x}
	\end{equation}
	in particular, $\mathcal{T}_{x}$ must be completely contained in the
	future-pointing cone of $a$.\\
	Indeed, according to the mentioned assumption, for any $x \in M$, there exists at least one vector $y$ in the intersection  $\mathcal{T}_{x}\cap \mathcal{T}
	_{x}^{a}$ - which thus must satisfy $A(y)>0, \Psi(y)>0$. But, as both a $A$ and $\Psi$ are assumed to be smooth (hence, continuous) on $\mathcal{T}_{x}$, in order to change sign, they should pass through $0$. But, the vanishing of either $A$ or $\Psi$ entails $L=0$, which is in contradiction with the positiveness axiom for $L$ inside $\mathcal{T}$, therefore \eqref{signs_T_x} must hold throughout $\mathcal{T}_{x}$. \vspace{5pt}
	\item On the boundary $\partial \mathcal{T}_{x},$ $L$ can be continuously
	prolonged as to satisfy:%
	\begin{equation}
	~A=0~\ \ or~\ \ \Psi =0.  \label{signs_boundary_T}
	\end{equation}
\end{enumerate}

\vspace{3pt}

\begin{lemma}
	\label{lem:I}(\textbf{The domain of }$s$\textbf{):} In any $\left( \alpha
	,\beta \right) $-metric Finsler spacetime, the set of values $s$ such that $%
	s(\dot{x})\in \mathcal{T}_{x}$ is an interval%
	\begin{equation}
		I\subset \lbrack s_{0},\infty ),~\ \ \ \ \ \ s_{0}:=\max \{ \left\langle
		b,b\right\rangle ,0\} .  \label{def_I}
	\end{equation}
\end{lemma}

\begin{proof}
	Let us show first that the set $I=\left\{ s(\dot{x})|~\dot{x}\,\in \mathcal{T}_{x}\right\}$ is contained in $[s_{0},\infty )$, that is:%
	\begin{equation*}
		s(\dot{x})\geq s_{0},~\ \ \ \ \forall \dot{x}\in \mathcal{T}_{x}.
	\end{equation*}
	
	If $b$ is $a$-timelike, that is, $\left\langle b,b\right\rangle >0,$ then,
	taking into account that $\dot{x}$ is, by definition, also $a$-timelike, the
	reverse Cauchy-Schwarz inequality tells us that: $\left\langle
	b,b\right\rangle \left\langle \dot{x},\dot{x}\right\rangle \leq \left\langle
	b,\dot{x}\right\rangle ^{2},$ which, using (\ref{AB_scalar_prod}) is nothing
	but:%
	\begin{equation*}
		\left\langle b,b\right\rangle A\leq B^{2},
	\end{equation*}%
	that is,%
	\begin{equation*}
		s\geq \left\langle b,b\right\rangle\ \ \  \Rightarrow \ \ \  s\geq \max \{ \left\langle
		b,b\right\rangle ,0\} =s_{0}.
	\end{equation*}
	
	In the case when $\left\langle b,b\right\rangle \leq 0,$ the statement is
	trivially satisfied, as $s=\dfrac{B^{2}}{A}$ is the
	ratio of two quantities that are nonnegative on  $\mathcal{T}_{x}$, hence $s\geq 0=s_{0}$.
	
	Further, since on $\mathcal{T}_{x},$ $A\not=0,$ the rational function $s=s(%
	\dot{x}):\mathcal{T}_{x}\rightarrow \mathbb{R}$ is continuous; using the
	connectedness axiom for $\mathcal{T}_{x},$ it follows that $I=s(\mathcal{T}%
	_{x})\subset \mathbb{R}$ must also be connected, namely, it is an interval.
\end{proof}

\begin{figure}[h!]
  \centering
  \begin{subfigure}[b]{0.25\linewidth}
    \includegraphics[width=\linewidth]{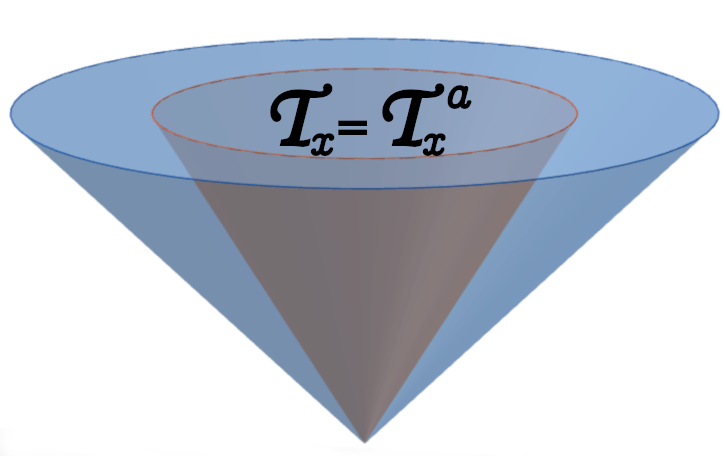}
     \caption{$\partial \mathcal T_x : \{ A=0\}$ }
  \end{subfigure}
  \begin{subfigure}[b]{0.25\linewidth}
    \includegraphics[width=\linewidth]{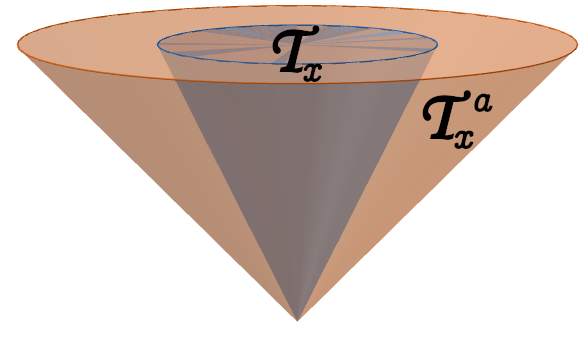}
    \caption{$\partial \mathcal T_x : \{ \Psi=0\}$ }
  \end{subfigure}
  \begin{subfigure}[b]{0.25\linewidth}
    \includegraphics[width=\linewidth]{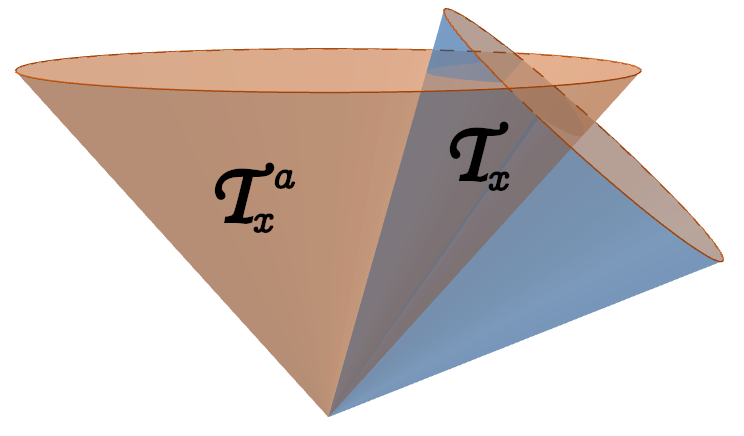}
    \caption{$\partial \mathcal T_x : \{ A=0\} \text{or} \{ \Psi=0\}$ }
  \end{subfigure}
  \begin{subfigure}[b]{0.20\linewidth}
   \includegraphics[width=\linewidth]{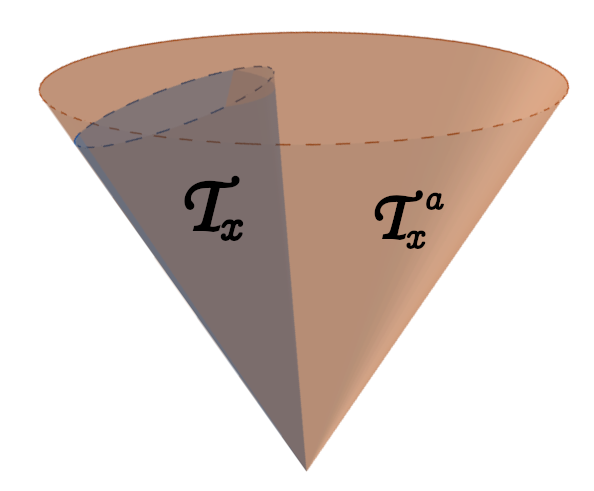}
   \caption{$\partial\mathcal T_x \cap \partial\mathcal T_x^a= \ell $ }
 \end{subfigure}
   \caption[Figure 1.]{The relative positions of the spacetime cones }
   \label{fig:position_cones}
\end{figure}

\vspace{5pt}
\noindent\textbf{Remark:\ sharpness (or non-sharpness) of the bounds for }$s$.

To establish precisely the interval $I$, we first need the critical points of the function $s=s(\dot{x})$. These turn out to be situated:
	\begin{itemize}
		\item In the plane $B=0$; in this case, the corresponding critical value is $s=0$.
		\vspace{5pt}		
		
		\item On the ray directed by $\dot{x}=\left\langle b,b\right\rangle$ emanating from the origin; these yield the critical value $s=\left\langle b,b\right\rangle$.
	\end{itemize}
 The above statement is justified as follows. Differentiating the expression $s=\frac{B^{2}}{A}$, we find that the condition $\dot{\partial}_{i}s=0$ is equivalent to:
	\begin{equation}
	B(b_{i}A-Ba_{ij}\dot{x}^{j})=0;
	\end{equation}
the vanishing of the first factor means precisely $B=0$, whereas the second one gives, after raising indices with the help of $a^{jk}$, that $\dot{x}$ is proportional to $\tilde{b}$.\\
Taking into account the above Lemma, we note that the corresponding critical values are -- if attained inside $\mathcal{T}_{x}$ -- \textit{minimal} values for $s$. This way, we find:
\begin{enumerate}
	\item The lower bound $s=s_{0}$ is attained for $\dot{x}\in \mathcal{T}_{x}$ in two situations:
	
	\begin{itemize}
		\item[i.] When $b\in \mathcal{T}_{x}$ is $L$-timelike (which implies that, in
		particular, it is also $a$-timelike, meaning that $\left\langle b,b\right\rangle >0$) and 
		$\dot{x}$ is collinear to $\tilde{b}$; \vspace{10pt}
		\item[ii.] When the critical hyperplane $B=0$ for $s$ intersects $\mathcal{T}_{x}$ (a necessary condition for this is that $b$ is $a$-spacelike).
		
	\end{itemize}
 	
	\vspace{10pt}
	\item For the upper bound, we have two possibilities:\ 
	\begin{itemize}
		\item[i.] If the boundary $\partial \mathcal{T}_{x}$ contains points where $A=0, B\neq 0,$ (as in Figure \ref{fig:position_cones} (a), (c) and (d)),  then, approaching these points, we will have $s\rightarrow \infty
		;$ therefore, in this case, $s\in \lbrack s_{0},\infty ),$ where the upper
		bound is sharp. \vspace{10pt}
	
		\item[ii.] If $A$ does not vanish anywhere on $\partial \mathcal{T}_{x}$ (situation (b) in Figure \ref{fig:position_cones}, when $\partial \mathcal{T}_{x}$ consists only of points where $\Psi =0,$ $%
		A>0$), then -- since obviously we cannot have $A=0$ \textit{inside }$\mathcal{T}%
		_{x}$ either -- it follows that $s$ has a finite supremum on $\mathcal{T}_{x}$%
		, in other words, we can safely write $s\in \lbrack s_{0},s_{1}),$ for some
		finite value $s_{1}>0.$
	\end{itemize}
\end{enumerate}

\bigskip

Having determined the domain of $s$, the next step for determining the precise conditions relating $\Psi $ and $\left\langle
b,b\right\rangle $ such that $L$ defines a spacetime structure, is to find the sign of the coefficient $%
\Psi -s\Psi ^{\prime }$ in (\ref{g_ij_Phi}).

\begin{lemma}
	\label{lem:sign_Psi}In any $\left( \alpha ,\beta \right) $-metric Finsler
	spacetime and at any $\dot{x}\in \mathcal{T}_{x},$ there holds:%
	\begin{equation*}
		\Psi -s\Psi ^{\prime }>0.
	\end{equation*}
\end{lemma}

\begin{proof}
	Let $I=s(\mathcal{T}_{x})\subset \lbrack s_{0},\infty )$ be the interval
	defined above. We will proceed in three steps:

\vspace{5pt}		
\textit{Step 1. There exists at least one }$\dot{x}$ \textit{on the boundary} $\partial \mathcal{T}_{x},$ \textit{such that }$B(\dot{x})\not=0:$

Pick an arbitrary $v\in \partial \mathcal{T}_{x}.$ If $v$ does not belong to
the hyperplane $H:=\left\{ v\in T_{x}M~|~B(v)=0\right\} $, then, the
statement is proven with $\dot{x}=v.$

Hence, in the following, let us assume that $v\in \partial \mathcal{T}%
_{x}\cap H.$

The cone $\mathcal{T}_{x}$ is, by definition, open (in the topology of $%
T_{x}M\simeq \mathbb{R}^{4}$), therefore it cannot be entirely contained in
the hyperplane $H.$ This way, there exists a $u\in \mathcal{T}_{x}$ such
that $B(u)\not=0$. We will show that, going along the line:%
\begin{equation*}
\ell :=\{u-\lambda v~|~\lambda \in \mathbb{R}\}
\end{equation*}%
we must encounter a boundary point, which is not in $H:$

\vspace{3pt}		
- We first note that $\ell $ has no common points with $H.$ Indeed, for any $%
\lambda \in \mathbb{R},$ we have $B(u-\lambda v)=B(u)-\lambda B(v)=B(u)\not=0.$

\vspace{3pt}	
- Second, $\ell \cap \mathcal{\overline{T}}_{x}$ is a non-empty, connected set. Indeed, $%
\ell $ contains an interior point of $\mathcal{T}_{x},$ which is $u;$ moreover, since $%
\mathcal{T}_{x}$ is convex, this means that its closure - which is
necessarily convex, too - must intersect $\ell $ by a segment, a half-line,
or the whole $\ell $. In any case, $\ell \cap \mathcal{\overline{T}}_{x}$ is
connected.

\vspace{3pt}	
- Third, we show that the set $\mathcal{\overline{T}}_{x}\cap \ell $ intersects the cone $%
\mathcal{T}_{x}^{a}:A=0.$ To this aim, let us build the function $f:\mathbb{R%
}\rightarrow \mathbb{R},$%
\begin{equation*}
f(\lambda ):=A\left( u-\lambda v\right) =\left\langle u,u\right\rangle
-2\lambda \left\langle u,v\right\rangle +\lambda ^{2}\left\langle
v,v\right\rangle .
\end{equation*}%
This function always has at least one root $\lambda _{0},$ as follows. 
\textit{(i)} If $\left\langle v,v\right\rangle =0$, that is, $v\in \partial 
\mathcal{T}_{x}^{a},$ then, since $u$ is $a$-timelike, it cannot be $a$%
-orthogonal to $v,$ which means $\left\langle u,v\right\rangle \not=0;$
hence, in this case $f$ is of first degree in $\lambda $ and thus has one
zero. \textit{(ii) }If $\left\langle v,v\right\rangle >0$, that is, $v\in 
\mathcal{T}_{x}^{a},$ then $f$ is quadratic, with halved discriminant:\ $\Delta
=\left\langle u,v\right\rangle ^{2}-\left\langle u,u\right\rangle
\left\langle v,v\right\rangle \geq 0,$ by virtue of the reverse
Cauchy-Schwarz inequality for $a$ - and hence, again, has real roots.

These roots $\lambda _{0}$ correspond to points $\dot{x}=u-\lambda _{0}v\in
\partial \mathcal{T}_{x}^{a}.$

\vspace{3pt}	
- Finally, taking into account the connectedness of  $\mathcal{\overline{T}}_{x}\cap \ell $, we find that, moving away from $u$ along the line $\ell,$ we stay in $\mathcal{T}_{x}$ and, at some
point, we must hit the boundary $\partial \mathcal{T}_{x};$ if, in a worst case
scenario, we do not hit first a point where $\Psi =0$, then, we must anyway reach a root $\lambda _{0}$ of $f$ - which will thus give
a boundary point $\dot{x}$ for $\mathcal{T}_{x}\subset \mathcal{T}_{x}^{a}.$
Moreover, since $\ell \cap H=\varnothing ,$ at this point we always have $B(%
\dot{x})\not=0.$

	\vspace{7pt}
	\noindent \textit{Step 2. There exists a subset }$I_{0}\subset I,$ \textit{on which }$%
	\Psi -s\Psi ^{\prime }>0$\textit{:}
	
	Pick $\dot{x}\in \partial \mathcal{T}_{x}$ such that $B(\dot{x})\not=0;$
	then, on a small enough neighborhood $V\subset \mathcal{T}_{x}$ of $\dot{x}%
	, $ $B$ is still nonvanishing, that is, $s\not=0.$ Then, on the set $s(V),$
	we can write:%
	\begin{equation}
		L=A\Psi =B^{2}\dfrac{\Psi }{s},  \label{L_B^2}
	\end{equation}%
	where the function $s\mapsto \dfrac{\Psi (s)}{s}$ is well defined, smooth
	and strictly positive on $s(V).$ But, on the other hand, as we approach $%
	\dot{x}\in \partial \mathcal{T}_{x},$ the function $L$ must tend to $0.$ Since $B(\dot{x}%
	) $ cannot vanish, the one which has to vanish is $\underset{s\rightarrow s(%
		\dot{x})}{\lim }\dfrac{\Psi (s)}{s},$ meaning that $\dfrac{\Psi }{s}$ must
	be strictly decreasing on some interval $I_{0}\subset I.$ The statement
	then follows from noticing that:%
	\begin{equation*}
		\Psi -s\Psi ^{\prime }=-s^{2}\dfrac{d}{ds}\left( \dfrac{\Psi }{s}\right)
		>0,~\ \ \ \ \forall s\in I_{0}.
	\end{equation*}
	
	\vspace{5pt}
	\noindent\textit{Step 3: }$\Psi -s\Psi ^{\prime }\not=0$ \textit{must have a constant
		on} $I:$
	
	Assuming that there exists some $s\in I$ such that $\Psi -s\Psi ^{\prime
	}=0, $ then, at the corresponding vectors $\dot{x}\in \mathcal{T}_{x},$ we
	would have, by (\ref{g_ij_Phi}): $g_{ij}(\dot{x})=\Psi ^{\prime }b_{i}b_{j}+%
	\dfrac{1}{2}A^{2}\Psi ^{\prime \prime }s_{\cdot i}s_{\cdot j}.$ But, in 4
	dimensions, such a matrix is immediately seen to be degenerate, which is not
	acceptable. Thus, $\Psi -s\Psi ^{\prime }\not=0$ on $I,$ which, together
	with the connectedness of $I$ and $\Psi -s\Psi ^{\prime }>0$ on $%
	I_{0}\subset I,$ yields the statement.
\end{proof}

The above Lemma allows us to state a simple condition such that, for all $%
\dot{x}\in \mathcal{T}_{x},$ $g_{\dot{x}}$ has Lorentzian signature $\left(
+,-,-,-\right) $ (thus agreeing with the sign of $L$).

\begin{lemma}
	\label{lem:sign_det_g} Assume $L=A\Psi :\mathcal{A}\rightarrow \mathbb{R}$
	is a pseudo-Finsler $\left( \alpha ,\beta \right) $-metric structure, $x\in
	M $ is an arbitrarily fixed point and $\mathcal{T}_{x}\subset \mathcal{A}%
	_{x} $ is a conic set satisfying (\ref{signs_T_x}), (\ref{signs_boundary_T}%
	). Then, at any $\dot{x}\in \mathcal{T}_{x},$ the following statements are
	equivalent:
	\begin{enumerate}
		\item[(i)] $g_{\dot{x}}$ has $(+,-,-,-)$ signature and is negative definite
		on the $g_{\dot{x}}$-orthogonal complement of $\dot{x};$ \vspace{5pt}
		
		\item[(ii)] $\det g_{\dot{x}}<0.$
	\end{enumerate}
\end{lemma}

\begin{proof}
	$(i)\rightarrow (ii)$ is obvious.
	
	$\quad(ii)\rightarrow (i):$ Assuming $\det (g_{\dot{x}})<0,$ the signature of $g_{%
		\dot{x}}$ can be either $\left( +,-,-,-\right) $, or $\left( -,+,+,+\right)$. 
	Using the $\left( +,-,-,-\right) $ signature of $a,$ we will show that, actually, the latter situation is not possible, as it would entail that in any
	diagonal form, $g_{\dot{x}}$ must have at least two minus signs:
	
	Fix an arbitrary $L$-timelike vector $\dot{x}\in \mathcal{T}_{x},$ which is not 
	collinear to $\tilde{b}$;  since $\mathcal{T}_{x} \subset \mathcal{T}_{x}^{a}$, the vector $\dot{x}$ must then be also timelike with 
	respect to $a.$
	Let us construct a $g_{\dot{x}}$-orthogonal basis $\mathcal{B}=\left\{
	e_{0},e_{1},e_{2},e_{3}\right\} $ as follows. Pick $e_{0}:=\dot{x};$ then,%
	\begin{equation*}
		g_{\dot{x}}(e_{0},e_{0})=g_{\dot{x}}(\dot{x},\dot{x})=L>0.
	\end{equation*}
	
As $e_{1},e_{2}$, we will choose any two mutually perpendicular vectors in the (2-dimensional) 
	$a$-orthogonal complement of $Span\{\tilde{b},\dot{x}\};$ this means:%
	\begin{equation}
		\left\langle \tilde{b},v\right\rangle =0,~~\left\langle \dot{x}%
		,v\right\rangle =0,~\ \ \ \ v\in \left\{ e_{1},e_{2}\right\} .
		\label{bx_orthog}
	\end{equation}%
With this choice, we get:
	\begin{eqnarray*}
		g_{\dot{x}}(\dot{x},v) &=&g_{ij}\left( \dot{x}\right) \dot{x}^{i}v^{j}=%
		\dfrac{1}{2}L_{\cdot i}(\dot{x})v^{i}=\dfrac{1}{2}\left( A_{\cdot i}\Psi
		+A\Psi ^{\prime }s_{\cdot i}\right) v^{i} \\
		&=&\left\langle \dot{x},v\right\rangle \Psi +\dfrac{1}{2}A\Psi ^{\prime
		}\left( s_{\cdot i}v^{i}\right) ,
	\end{eqnarray*}%
	where we have used:\ $A_{\cdot i}v^{i}=2a_{ik}\dot{x}^{k}v^{i}=2\left\langle 
	\dot{x},v\right\rangle ;$ further, using (\ref{def_alpha_beta_L}), we get: $s_{\cdot i}=%
	\dfrac{1}{A}(2Bb_{i}-sA_{\cdot i}),$ which, using (\ref{bx_orthog}) gives:%
	\begin{equation*}
		s_{\cdot i}v^{i}=\dfrac{1}{A}\left( 2B\left\langle \tilde{b},v\right\rangle
		-2s\left\langle \dot{x},v\right\rangle \right) =0
	\end{equation*}%
	and finally, 
	\begin{equation*}
		g_{\dot{x}}(\dot{x},v)=0,
	\end{equation*}%
	that is, $e_{1}$ and $e_{2}$ are indeed, $g_{\dot{x}}$-orthogonal to $\dot{x}$.
	
	It remains to check the sign of $g_{\dot{x}}(v,v)=g_{ij}\left( \dot{x}%
	\right) v^{i}v^{j},$ for $v\in \left\{ e_{1},e_{2}\right\} .$ Substituting $%
	g_{ij}$ from (\ref{g_ij_Phi}) and taking into account that $b_{i}v^{i}=0,$ $%
	s_{\cdot i}v^{i}=0$, we find:%
	\begin{equation*}
		g_{\dot{x}}(v,v)=(\Psi -s\Psi ^{\prime })\left\langle v,v\right\rangle
		.
	\end{equation*}%
	But, on the one hand, the assumption $\left\langle \dot{x},v\right\rangle =0$
	implies that $v$ must be $a$-spacelike, that is, $\left\langle
	v,v\right\rangle <0$ and, on the other hand, using Lemma \ref{lem:sign_Psi},
	the first factor above is strictly positive. All in all, we get:%
	\begin{equation*}
		g_{\dot{x}}\left( v,v\right) <0,~\ \ \ v\in \{e_{1},e_{2}\}
	\end{equation*}%
	and $g_{\dot{x}}(\dot{x},\dot{x})>0.$ 
	Then, for any choice of the fourth basis vector $e_{3}$ in the $g_{\dot{x}}$-orthogonal complement of $e_{0}, e_{1}, e_{2}$, we find, using the hypothesis that $\det (g_{\dot{x}})<0$:
	\begin{equation*}
		g_{\dot{x}}(e_{3},e_{3})<0,
	\end{equation*}%
	which proves \textit{(i).}
\end{proof}

\bigskip

Further, let us introduce the function $\boldsymbol{\sigma =\boldsymbol\sigma }%
(s):I\rightarrow \mathbb{R}$ (with $I$ as in Lemma \ref{lem:I}), as:%
\begin{equation}
	\boldsymbol{\sigma }:=\dfrac{\left( \Psi -s\Psi ^{\prime }\right) ^{2}}{\Psi }.
	\label{def_sigma}
\end{equation}%
A direct computation (see the Appendix \ref{app:detg}) then proves the following
Proposition.

\begin{proposition}
	For any pseudo-Finsler function $L=A\Psi (s),$ $s=\dfrac{B^{2}}{A},$ the
	determinant of its Finslerian metric tensor $g_{\dot{x}}$ is:
	
	\begin{equation}
		\det (g_{\dot{x}})=\Psi ^{2}(\Psi -s\Psi ^{\prime })\det (a)\dfrac{d}{ds}\left[ (s-\left\langle b,b\right\rangle )\boldsymbol{%
			\sigma }\right] .  \label{det(g)}
	\end{equation}
\end{proposition}

\bigskip

Using the expression of $\det (g)$ and the above Lemmas, we are now able to
prove the main result of this section.

\begin{Theorem}
	\label{thm:sign_g}\textbf{(The spacetime conditions): }Let $M$ be a 4-dimensional connected, orientable manifold. An $(\alpha,\beta)$-metric function $L:\mathcal{A}\rightarrow \mathbb{R}$, $L=A\Psi(s),$ $s=\dfrac{B^{2}}{A},$ with the underlying pseudo-Riemannian metric $A$ of Lorentzian signature $(+,-,-,-)$, defines a Finsler spacetime structure, if and only if there exists a conic subbundle $\mathcal{T}\subset \mathcal{A}$ with connected fibers $\mathcal{T}%
	_{x}$, obeying the following conditions:
	\begin{enumerate}
		\item[i)] $A>0,\Psi >0$ on $\mathcal{T}_{x}$ and $\underset{\dot{x}%
			\rightarrow \partial \mathcal{T}_{x}}{\lim }(A\Psi )=0.$ \vspace{5pt}
		
		\item[ii)] For all values $s$ corresponding to vectors $\dot{x}\in \mathcal{T%
		}_{x}:$%
		\begin{equation}
			\Psi -s\Psi ^{\prime }>0,~\ \ \ \ \ \left( s-\left\langle b,b\right\rangle
			\right) \dfrac{d}{ds}\ln \boldsymbol{\sigma }>-1.  \label{spacetime_cond}
		\end{equation}
	\end{enumerate}
\end{Theorem}

\begin{proof}
	$\rightarrow :$ Assuming that $\left( M,L\right) $ is a Finsler spacetime,
	then, by definition, on each of its future-pointing timelike cones $\mathcal{T}_{x},$ $x\in M,$ there must hold $i)$.
	The first inequality $ii)$ follows from Lemma \ref{lem:sign_Psi}. Then, using this inequality together with $\Psi>0$, $\det g_{\dot{x}}<0$, $\det a<0$ into the expression (\ref{det(g)}) for $\det g_{\dot{x}}$, we get:
	\begin{equation*}
		\dfrac{d}{ds}\left[ (s-\left\langle b,b\right\rangle )\boldsymbol{\sigma }\right]
		=:\dfrac{d}{ds}\left( \rho \boldsymbol{\sigma }\right) >0,
	\end{equation*}%
	where $\rho (s):=s-\left\langle b,b\right\rangle $ is known by Lemma \ref%
	{lem:I} to obey $\rho \geq 0.$ Besides, $\dfrac{d\rho }{ds}=1,$ which means
	that the above inequality can be re-expressed as:%
	\begin{equation*}
		\dfrac{d}{ds}(\rho \boldsymbol{\sigma })=\rho ^{\prime }\boldsymbol{\sigma }+\rho 
		\boldsymbol{\sigma }^{\prime }=\boldsymbol{\sigma }+\rho \boldsymbol{\sigma }^{\prime
		}>0;
	\end{equation*}%
	using $\boldsymbol{\sigma }>0,$ this is in turn equivalent to:%
	\begin{equation*}
		\rho \dfrac{\boldsymbol{\sigma }^{\prime }}{\boldsymbol{\sigma }}>-1,
	\end{equation*}%
	which is precisely the second inequality (\ref{spacetime_cond}).
	
	\vspace{5pt}
	\noindent $\leftarrow :$ Assume now that there exists a conic subbundle $\mathcal{T} \subset \mathcal{A}$ with connected fibers and obeying conditions $i)$ and $ii).$ Then, in
	order to prove that $(M,L)$ is a Finsler spacetime, it is sufficient to show
	that $g_{\dot{x}}$ has $\left( +,-,-,-\right) $ signature for all $\dot{x}%
	\in \mathcal{T}_{x}$ - which, using Lemma \ref{lem:sign_det_g}, is the same
	as: $\det g_{\dot{x}}<0.$ But, using $\Psi >0,$ $\Psi -s\Psi ^{\prime }>0$,
	this reduces to:%
	\begin{equation*}
		\dfrac{d}{ds}\left[ (s-\left\langle b,b\right\rangle )\boldsymbol{\sigma }\right]
		>0.
	\end{equation*}%
	 But, the latter is equivalent to: $\boldsymbol{\sigma }+\left( s-\left\langle
	b,b\right\rangle \right) \boldsymbol{\sigma }^{\prime }>0,$ which, by virtue of $%
	\boldsymbol{\sigma }=\dfrac{\left( \Psi -s\Psi ^{\prime }\right) ^{2}}{\Psi }>0,$
	is nothing but the second inequality (\ref{spacetime_cond}) - and thus,
	satisfied by hypothesis.
\end{proof}

\bigskip

\noindent\textbf{Remarks:}
\begin{enumerate}
	\item Since $s-\left\langle b,b\right\rangle \geq 0,$ then any
	strictly positive and monotonically increasing function $\boldsymbol{\sigma }$
	will obey our condition, as, in this case, we will have the stronger
	inequality: $\left( s-\left\langle b,b\right\rangle \right) \dfrac{d}{ds}\ln 
	\boldsymbol{\sigma }\geq 0.$
	\vspace{5pt}		
	
	\item The above Theorem also works if one allows $\Psi$ to be of class $\mathcal{C}^{1}$ only along one direction in each $\mathcal{T}_{x}$, as in \cite{Caponio-Stancarone, Caponio-Masiello} - with the only mention that the second condition \eqref{spacetime_cond} (which involves  second derivatives of $\Psi$) will only hold outside that direction.
\end{enumerate}

\bigskip 
\section{Examples}\label{sec:ex}


\subsection{Lorentzian metrics $L=\kappa A$}
	These are given by 
	\begin{equation}
		\Psi =\kappa ,  \label{example_Lorentzian}
	\end{equation}%
	where $\kappa >0$ is a constant - and they are always Finsler spacetimes in
	the sense of the above Definition. Indeed, in this case, we obtain $\Psi
	-s\Psi ^{\prime }=\kappa >0$, $~\ \ \boldsymbol{\sigma }=\kappa >0\ $ and $\ \left(
	s-\left\langle b,b\right\rangle \right) \dfrac{\boldsymbol{\sigma }^{\prime }}{%
		\boldsymbol{\sigma }}=0,$ hence conditions (\ref{spacetime_cond}) are trivially
	satisfied; the light cones of $L$ are obviously the light cones of $A.$
	\bigskip

	\subsection{Randers metrics $L=\epsilon (\sqrt{A}+B)^{2},\,\, 	\epsilon =sign(\sqrt{A}+B)$}  
	Before determining the corresponding function $\Psi$, some preliminary remarks on the future-pointing timelike cones $\mathcal{T}_{x}, x \in M$, of a Randers spacetime, are in place:
	\begin{itemize}
		\item  The cone $\mathcal{T}_{x}$ is a connected component of the conic set $\sqrt{A}+B>0$, more precisely, the connected component of the cone
		\begin{equation}
		(a_{ij}-b_{i}b_{j})\dot{x}^i\dot{x}^{j} > 0 \label{interior_Randers}
		\end{equation}
		lying in the future-pointing cone $\mathcal{T}_{x}^{a}$. This is seen as $A=0$ -- which is a singularity for $L$ -- cannot occur within the cone \eqref{interior_Randers}. Hence, the boundary $L=0$ must be a "sharper" cone than $A=0$, see also \cite{HPV2019}.
	
		\vspace{5pt}
		\item On the boundary $\partial 
		\mathcal{T}_{x},$ we must have $L=0$, which implies:
		\begin{equation}
		B=-\sqrt{A}\leq 0. \label{cond_Randers}
		\end{equation}
	This leads to the even more interesting conclusion below.
		\vspace{5pt}
		\item At any point $x \in M$ of a Randers spacetime, \textit{the  whole cone} $\mathcal{T}_{x}$ lies in the half-space $B \leq 0$:
		\begin{equation}
			B(\dot{x})\leq 0,\, \forall \dot{x} \in \mathcal{T}_{x}.  \label{Randers_ineq_B}
		\end{equation}
		To justify the above statement, assume that the intersection $\mathcal{C}=\mathcal{T}_{x} \cap \{B>0\}$ is non-empty. Consequently, it is an open conic and convex set and, moreover, by continuity, on the boundary $\partial \mathcal{C}$, we have $B\geq 0$. Then, proceeding as in the first step of the proof of Lemma \ref{lem:sign_Psi}, we find that there must exist a point $\dot{x} \in \partial\mathcal{C}$ where $B\neq 0$; such a point is, thus, a point of $\partial\mathcal{T}_{x}$ -- where, in addition, $B>0$,  which is in contradiction with the above remark. Therefore, $B$ cannot be positive inside $\mathcal{T}_{x}$.
	\end{itemize}

The last remark points out that, inside the future-pointing timelike domain for Randers spacetimes, we always have: $\sqrt{s}=-\frac{B}{\sqrt{A}}$. In other words, the restriction $\Psi:\mathcal{T} \rightarrow \mathbb{R}$ (which is relevant to our purposes) has the expression:
\begin{equation*}
\Psi =(1-\sqrt{s})^{2}.
\end{equation*}

The quantities in Theorem \ref{thm:sign_g} are then immediately obtained as:
\begin{equation}
\Psi - s\Psi'= 1-\sqrt{s},\qquad \sigma =1. \label{Randers_aux}
\end{equation}

We are now able to prove the following result.
	\begin{proposition}\label{prop:rand}
      A Randers-type deformation of a 4-dimensional Lorentzian metric $a$ defines a Finsler spacetime structure if and only if 
      \begin{equation*}
      0\leq\left\langle b,b\right\rangle<1. \label{Randers_cond}
      \end{equation*}
    If this is the case, then each of its future-pointing timelike cones $\mathcal{T}_{x}$ is the intersection of the cone $A-B^{2}>0$ with future-pointing cones of $a$.
  \end{proposition}

	\begin{proof}
	$\rightarrow$: Assume that $L$ defines a spacetime structure.
		The fact that, in this case, the cones $\mathcal{T}_{x}$ are the intersections of the cones $\sqrt{A}>-B$ (which is the same as $A>B^{2}$, taking into account that $B<0$) with $\mathcal{T}_{x}$ was already discussed in the first remark above.\\
		 Further, according to Theorem $\ref{thm:sign_g}$, on its future-pointing timelike cones, we must have: $\Psi - s\Psi'= 1-\sqrt{s}>0$, which is equivalent to: $s<1$. But, using Lemma \ref{lem:I}, this gives	
		 \begin{equation*}
		\left\langle b,b\right\rangle \leq  s<1.
		\end{equation*} 
	To prove the first inequality $\left\langle b,b\right\rangle\geq 0$, let us assume that this is not the case, namely, $\left\langle b,b\right\rangle <0$. We will show that, in this case, the hyperplane $B=0$ must intersect $\mathcal{T}_{x}$. 
	To this aim, fix an arbitrary point $x \in M$ and an $a$-orthonormal basis $\{e_{0},e_{1},e_{2},e_{3}\}$, with $e_{0}$ being $a$-timelike and $e_{1}$ collinear to $\tilde{b}$. It follows that $\tilde{b}=\lambda e_{1}$ for some $\lambda \in \mathbb{R}^{*}$. Accordingly, $A=0$ becomes equivalent to $(\dot{x}^{0})^{2}-(\dot{x}^{1})^{2}-(\dot{x}^{2})^{2}-(\dot{x}^{3})^{2}=0$ and the hyperplane $B=0$ is described by the equation $\dot{x}^{1}=0$. But then, using \eqref{interior_Randers}, it follows that the intersection between this hyperplane and $\mathcal{T}_{x}$ is the region of the 3-dimensional cone:
	\begin{equation}
	(\dot{x}^{0})^{2}-(\dot{x}^{2})^{2}-(\dot{x}^{3})^{2}>0,
	\end{equation}
	lying in $\mathcal{T}_{x}^{a}$ (that is, with $\dot{x}^{0}>0$), which is by far non-empty. This is in contradiction with the fact that the hyperplane $B=0$ cannot intersect $\mathcal{T}_{x}$.
	We conclude that our assumption was false, that is, $\left\langle b,b\right\rangle\geq 0$. \\
	
	$\leftarrow$: Assume now that $0 \leq \left\langle b,b\right\rangle <1$ and let us show that $L$ satisfies the conditions of Theorem $\ref{thm:sign_g}$. 
	Define $\mathcal{T}$ as the intersection of the future-pointing cone bundle $\mathcal{T}^{a}$ with the set $A-B^{2}>0$, that is:
	\begin{equation}
	(a_{ij}-b_{i}b_{j})\dot{x}^i\dot{x}^{j} > 0;
	\end{equation}
	in particular, this entails: $s \in (0,1)$.
	The latter inequality makes it clear that, at each point $x \in M$, the boundary $\partial\mathcal{T}_{x}$ is a quadric; using the hypothesis $\left\langle b,b\right\rangle <1$ and the Lorentzian signature of $a$, we find that:
	\begin{equation*}
	det(a_{ij}-b_{i}b_{j}) = det(a_{ij})(1-\left\langle b,b\right\rangle) <0.
	\end{equation*}  
	This way, the matrix $(a_{ij}-b_{i}b_{j})$ can only have $(+,-,-,-)$ or $(-,+,+,+)$ signature. Moreover, regarding its entries as continuous functions of $b$, we find out that, for $b=0$, the signature of our matrix coincides with the one of $a$, as the matrices themselves coincide. Since the signature cannot change without passing through a zero of the determinant, we find that for all $b$ in the given domain, we still have a $(+,-,-,-)$ signature. In other words, each set $\mathcal{T}_{x}$ is the interior of a convex (hence, connected) cone. 

	Moreover, on each fiber $\mathcal{T}_{x}, x \in M$, we have:
	\begin{itemize}
	\item The functions $A$ and $\Psi$ are, by definition, positive inside $\mathcal{T}$ and, when approaching the boundary of $\mathcal{T}$ (that is, for $s \rightarrow 1$), we find that $\Psi \rightarrow 0$, hence $L \rightarrow 0$.
	The smoothness of $\Psi$ follows as the value $s=0$ (which would correspond to $B=0$) cannot appear inside $\mathcal{T}_{x}$. This happens since $b$ is, by hypothesis, non-spacelike with respect to $a$, that is, $B=0$ cannot occur within $\mathcal{T}_{x}^{a}$ and accordingly, neither in $\mathcal{T}_{x}$.
	 \vspace{5pt}
	\item Conditions ii) of Theorem \ref{thm:sign_g} follow immediately using \eqref{Randers_aux}.
	\end{itemize}
	Thus, $L=A\Psi$ defines a Finsler spacetime structure.
	\end{proof}	
	
In Figure \ref{fig:Randers}, the possible relative positions of hyperplane $B=0$ with 
respect to the timelike cones $\mathcal T_x$ and $\mathcal T_x^a$ are displayed.
		
\begin{figure}[h!]
  \centering
  \begin{subfigure}[b]{0.3\textwidth}
    \includegraphics[width=\linewidth]{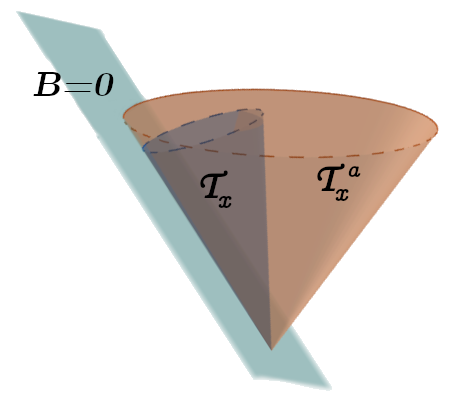}
    \caption{$b$ lightlike }
  \end{subfigure}
\hspace{1.8cm}
  \begin{subfigure}[b]{0.3\textwidth}
    \includegraphics[width=\linewidth]{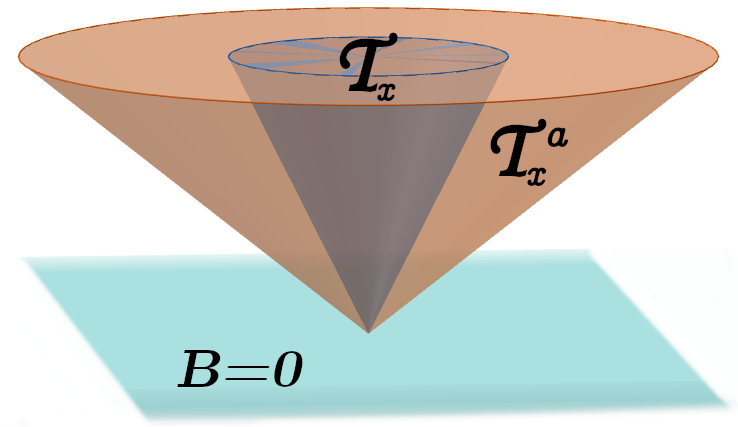}
    \caption{ $b$ timelike}
  \end{subfigure}
   \caption{Possible position of the hyperplane $B=0$ with respect to the spacetime cones.}	
   \label{fig:Randers}
\end{figure}	
	
	\textbf{Remark.} The result above extends the previous one in \cite{HPV2019}, by taking into consideration the case when $b$ is $a$-spacelike -- and actually proving that this \textit{never} defines a Randers spacetime.\\

\subsection{Bogoslovsky-Kropina metrics $L=A^{1-q}B^{q}$}
	We will leave aside the case when $q=0$, as it is trivial.
	\begin{proposition}\label{prop:bog}
	A Bogoslovsky-Kropina metric $L=A^{1-q}B^{q}$, with $q \neq 0$, defines a spacetime structure if and only if one of the following happens:
	\begin{enumerate}
		\item [(i)] $\left\langle b,b\right\rangle \geq 0$ and $	q\in \lbrack -1,1).$
		In this case, the future-pointing cones	of $L$ coincide with those of $A$.
		\vspace{5pt}
		\item [(ii)] $\left\langle b,b\right\rangle< 0$ and $	q\in (0,1)$.
		In this case, the future-pointing cones of $L$ are obtained by intersecting the future-pointing cones of $A$ with the half-space $B>0$.
	\end{enumerate}
	\end{proposition}
 	
\begin{proof}
	The Finsler function is expressed as $L=As^q$, which means that $\Psi$ has the expression:
	\begin{equation}
	\Psi(s):=s^{q}.
	\end{equation}
	$\rightarrow$ Assume that $L$ defines a spacetime structure and let us first identify the future-pointing cones $\mathcal{T}_{x}$ of $L$, together with the domain of definition of $\Psi$.\\ 
	Fix $x\in M$. As we have seen above, for any $(\alpha,\beta)$-metric, the cone $\mathcal{T}_{x}$ must be contained in $\mathcal{T}_{x}^{a}$. Moreover, in our specific case, boundary points must be either on the cone $A=0$, or in the plane $B=0$ Thus:
	\begin{enumerate}
		\item If $b$ is non-spacelike with respect to $a$, then $B=0$ cannot happen inside $\mathcal{T}_{x}^{a}$ (it can, in the worst case when $b$ is $a$-lightlike, happen on its boundary). In this case, we thus have $\mathcal{T}_{x}=\mathcal{T}_{x}^{a}$; moreover, the reverse Cauchy-Schwarz inequality tells us that $s\geq \left\langle b,b\right\rangle$, with equality for $\dot{x} \sim \tilde{b}$. In other words, the minimum value $s_{0}= \left\langle b,b\right\rangle$ is always attained on the closure of $\mathcal{T}_{x}$.
		\vspace{5pt}
	
		\item If $b$ is $a$-spacelike, then points with $B=0$ will always exist inside $\mathcal{T}_{x}^{a}$; we note that, in order to have a finite limit for $L=A^{1-q}B^{q}$ as we approach the hyperplane $B=0$, we must necessarily have:
		\begin{equation}
		q > 0. \label{q_0}
		\end{equation}
		In this case, the cone $\mathcal{T}_{x}$ will be the region of the cone $\mathcal{T}_{x}^{a}$ situated in the half-space $B>0$ and $s$ will tend to zero as we approach boundary points with $B=0$.
		\vspace{5pt}
		
		\item On the other hand, as shown above, the boundary of each cone $\mathcal{T}_{x}$ of any $(\alpha,\beta)$-metric spacetime must contain at least a point where $B \neq 0$; in our case, this means that there will always be at least a boundary point satisfying $A=0$, that is, we necessarily have in $\mathcal{T}_{x}$, values $s \rightarrow \infty$.
	\end{enumerate}
	Briefly: in any case, $\Psi$ is defined on the entire interval:
	\begin{equation}
	s\in (s_{0},\infty),\,\,\, s_{0}:=max\{\left\langle b,b\right\rangle,0\}.
	\end{equation} 
	Having identified the function $\Psi$, we are now ready to rewrite, in our case, the conditions in Theorem \ref{thm:sign_g}. \\
	(i): Since, inside the above identified cones $\mathcal{T}_{x}$, we obviously have $A>0, \Psi>0$ and  $\Psi$ - smooth, in order to check this conditions, it remains to impose that, approaching any point of the boundary $\mathcal{T}_{x}$, we should have $L=A^{1-q}B^{q} \rightarrow 0$. This immediately implies:
	\begin{equation}
	q<1.
	\end{equation}
	(ii). The first inequality $\Psi -s\Psi ^{\prime }>0$  becomes $s^{q}(1-q)>0$, that is, it is equivalent to the same inequality $q<1$.\\
	Noting that $\boldsymbol{\sigma }:=s^{q}(q-1)^{2},$ the second inequality (\ref{spacetime_cond}) reads: $(s-\left\langle b,b\right\rangle )\dfrac{q}{s}>-1$, which, taking into account that $s>0$ is equivalent to:
	\begin{equation}
	\left( q+1\right) s>\left\langle b,b\right\rangle q. \label{ineq_BK}
	\end{equation}
	For this inequality to hold for any $s \in (s_{0},\infty)$, it is necessary and sufficient that it holds (non-strictly) at the endpoints of this interval. Thus:
	\begin{itemize}
		\item If $\left\langle b,b\right\rangle \geq 0$, then at the lower bound $s_{0}=\left\langle b,b\right\rangle$, the (non-strict) inequality \eqref{ineq_BK} is identically satisfied. Imposing it for $s\rightarrow \infty$, we find:
		\begin{equation}
		q \geq -1,
		\end{equation}
		which proves the necessity of condition (i) in our statement.
		
		\vspace{5pt}
		\item If $\left\langle b,b\right\rangle < 0$, then, for $s\rightarrow \infty$, we find: $q>-1$, whereas the lower bound $s_{0}=0$ gives: $q\geq 0$, proving the necessity of condition (ii).
	\end{itemize}
$\leftarrow$: Assuming one of the situations (i) or (ii) happens, then the conditions in Theorem \ref{thm:sign_g} are immediately satisfied on the said cones by $\Psi:(s_{0},\infty)\rightarrow \mathbb{R}, \Psi(s)=s^q$, hence, $L$ defines a spacetime structure.
\end{proof}
\bigskip

\subsection{Kundt metrics}
	These are given by:
	\begin{equation}
L(x,\dot{x})=As^{-p}\left( k+ms\right) ^{p+1}=B^{-2p}\left( kA+mB^{2}\right)
^{p+1},  \label{def_Kundt}
\end{equation}%
where $p,k,m\in \mathbb{R}$ are constants. As $k=0$ gives a degenerate
metric and $m=0$ provides the already discussed Bogoslovsky-Kropina class,
we will assume here that $k,m\not=0.$ The corresponding function $\Psi $ is
given by:%
\begin{equation*}
\Psi (s)=\dfrac{\left( k+ms\right) ^{p+1}}{s^{p}}.
\end{equation*}

\begin{proposition}\label{prop:kundt}
The Kundt metric $L(x,\dot{x})=As^{-p}\left( k+ms\right) ^{p+1}$ defines a Finsler spacetime metric if and only if $\  \ k>0,\ \  m<0,$ and one of the following situations takes place:
\begin{enumerate}
	\item [(i)] $\left\langle b,b\right\rangle \geq 0\ \text{ and }\ p\in(-1,1]$.\\
	In this case, the future-pointing timelike cones of $L$ are the connected components of the cones $kA+mB^2>0$ contained in the future-pointing timelike cones $\mathcal{T}_{x}^{a}$.
	\vspace{5pt}
	\item [(ii)] $\left\langle b,b\right\rangle <0\ \text{ and }\ p\in(-1,0)$.\\
	In this case, the future-pointing timelike cones of $L$ are regions of the the cones $kA+mB^2>0$ lying both in the future-pointing timelike cones $\mathcal{T}_{x}^{a}$ and in the half-space $B>0$.
\end{enumerate}
\end{proposition}

\begin{proof}
$\rightarrow:$
Assume $L$ defines a Finsler spacetime structure. Let us first make some general remarks on the future-pointing timelike cones $\mathcal{T}_{x}$, that hold regardless of the causal character of $b$. To this aim, fix an arbitrary $x \in M$.

The latter equality in (\ref{def_Kundt}) reveals that the only boundary
points $\dot{x}\in \partial \mathcal{T}_{x}$ where $B\not=0$ are in the set $kA+mB^{2}=0,$ which corresponds to: $s=-\dfrac{k}{m}$. Since such boundary points must exist and lead to the finite value $L=0$ and, on the other hand, they must correspond to strictly positive values of $s$, we get, respectively:
\begin{equation}
p+1>0,\ \ \     -\dfrac{k}{m}>0.  \label{ratio_km_sign}
\end{equation}
The latter relation points out that $k$ and $m$ must have opposite signs. To see precisely what these signs are, we evaluate the first condition Theorem \ref{thm:sign_g} (ii):
\begin{equation}
\Psi-s\Psi ^{\prime }>0\quad \Leftrightarrow\quad k(p+1)(k+ms)^{p}s^{-p}>0, \label{Psi_prime_Kundt}
\end{equation}
which shows that only viable variant is:
\begin{equation}
k>0,\ \ \  m<0, \label{signs_k_m}
\end{equation}
as claimed. But, this means that the inequality $k+ms>0$ is equivalent to $s<-\dfrac{k}{m}$, which gives, for $\dot{x} \in \mathcal{T}_{x}$:
\begin{equation*}
s\in I \subset \left(0,-\dfrac{k}{m}\right),
\end{equation*}
where the upper bound $-\dfrac{k}{m}$ of the interval $I$ is sharp; the lower bound, yet, depends on the $a$-causal character of $b$.\\
The second condition in Theorem \ref{thm:sign_g} (ii) becomes:
\begin{equation}
\rho \dfrac{\boldsymbol{\sigma }^{\prime }}{\boldsymbol \sigma }>-1\quad\Leftrightarrow\quad \left(
\left\langle b,b\right\rangle m-kp+k\right) s+\left\langle b,b\right\rangle
kp>0. \label{Kundt_cond_sigma}
\end{equation}
This holds for every $s$ corresponding to $\dot{x} \in \mathcal{T}_{x}$ if and only if it holds (non-strictly) at the bounds of the interval $I$.\\
At the upper bound $s\rightarrow -\dfrac{k}{m}$, this gives:
\begin{equation*}
\left( p-1\right) \left( \left\langle b,b\right\rangle m+k\right) \leq 0.
\end{equation*}
But, taking into account $s\geq \left\langle b,b\right\rangle$ together with $m<0$, it follows that
\begin{equation}
k+m \left\langle b,b\right\rangle \geq k+ms>0, \label{ineq_bmk}
\end{equation}
therefore we must have $p-1 \leq 0$; all in all:
\begin{equation}
p \in (-1,1].
\end{equation}

Another important consequence of \eqref{ineq_bmk} is that it ensures that the cones $kA+mB^2>0$ are, indeed, convex cones. This is seen as the matrix $(k a_{ij}+m b_{i}b_{j})$ has the determinant $k^{4}\det(a)(1+\frac{m}{k}\left\langle b,b\right\rangle)<0$ - and a quick similar reasoning to the one in the Randers subsection shows that its signature is the same as the one of $a$, that is $(+,-,-,-)$. In other words, $kA+mB^2>0$ is the interior of a convex cone. Moreover, as $kA>kA+mB^2$, this cone is completely contained in the cone $A>0$.\ 

It remains to evaluate \eqref{Kundt_cond_sigma} at the lower bound for $s$. Since this depends on the sign of $\left\langle b,b\right\rangle$, we distinguish two major cases:
\begin{enumerate}
	\item Assume $\left\langle b,b\right\rangle \geq 0$. Then $B=0$ cannot happen inside the future-pointing timelike cones of $a$, hence, neither in the Finslerian ones $\mathcal{T}_{x}$. Thus, in this case, each of the cones $\mathcal{T}_{x}$ is the connected component of the convex cone:
	\begin{equation}
	kA+mB^2>0
	\end{equation}
	contained in $\mathcal{T}_{x}^{a}$.
	In this case, a quick check shows that:
	\begin{equation}
	(k a_{ij}+m b_{i}b_{j})\tilde{b}^{i}\tilde{b}^{j}=\left\langle b,b\right\rangle(k+m\left\langle b,b\right\rangle) \geq 0,
	\end{equation}
	which means that one of the vectors $\tilde{b}$ or $-\tilde{b}$ belongs to $\overline{\mathcal{T}}_{x}$. In either case, we find that the corresponding value $s_{0}=\left\langle b,b\right\rangle$ is attained on the closure of $\mathcal{T}_{x}$ - and it represents the lower bound for $s\in I$.\\
	At this lower bound, inequality \eqref{Kundt_cond_sigma} reduces to:
	\begin{equation*}
	\left\langle b,b\right\rangle(k+m\left\langle b,b\right\rangle)\geq 0	
	\end{equation*}
	and is identically satisfied, for any value of $p \in (-1,1]$.
	
	\vspace{5pt}
	\item Suppose $\left\langle b,b\right\rangle < 0$. Then the hyperplane $B=0$  has common points with the cone $kA+mB^2>0$. But, points with $B=0$ are either singular points for $L$ (if $p>0$), or null cone points if $p<0$. Obviously, the only viable situation is the second one, namely:
	\begin{equation}
	p \in (-1,0).
	\end{equation}
	Knowing this, evaluation of \eqref{Kundt_cond_sigma} at the lower bound of the interval for $s$, which is, in this case, $s=0$, gives:
	\begin{equation}
	\left\langle b,b\right\rangle kp \geq 0
	\end{equation}
	which always happens, as: $\left\langle b,b\right\rangle<0, p<0, k>0$.\\
	The future-pointing timelike cones of $L$ are then the intersections of the cone $kA+mB^2>0$ with the future-pointing cones $\mathcal{T}_{x}^{a}$ and the half-space $B>0$.
\end{enumerate} 
$\leftarrow:$ Assuming $k>0,\ m<0$, then, the matrix $(ka_{ij}+mb_{i}b_{j})$ has Lorentzian signature $(+,-,-,-)$, meaning that the set $kA+mB^2>0$ is a convex cone, which is, moreover, contained in the cone $A>0$. Then, assuming one of the situations (i), (ii) holds, one can immediately check the conditions of Theorem \eqref{thm:sign_g} for the specified $\mathcal{T}_{x}$ (note that statement (ii) of the theorem is equivalent to \eqref{Psi_prime_Kundt},\eqref{Kundt_cond_sigma}).

\end{proof}

\bigskip

\subsection{Exponential metrics $L=Ae^{P(s)}$}

	Here, $P=P(s)$ denotes an arbitrary smooth function for all $s$ in the interval  $I=(\max \{0,\left\langle b,b\right\rangle
	\},\infty )$.
	Since $e^{P(s)}$ has no zeros, the cone $\mathcal{T}_{x}$ is the entire cone of $a$ - and, accordingly, the domain of 
	$s$ is the entire interval $I$ above.
	\begin{proposition} {\label{prop:exp}}
	The exponential metric  $L=Ae^{P(s)}$ defines a Finsler spacetime structure if and
	only if:%
	\begin{eqnarray}
		&&\underset{s\rightarrow \infty }{\lim }\dfrac{e^{P(s)}}{s}=0,~
		\label{expo_conds} \\[5pt]
		&1-sP^{\prime }>0,~\ \ &1-sP^{\prime }>\left( s-\left\langle
		b,b\right\rangle \right) \left( s\left( P^{\prime }\right) ^{2}+2sP^{\prime
			\prime }+P^{\prime }\right) ,~\ \ \forall s\in I.  \notag
	\end{eqnarray}
	If this is the case, then the future-pointing timelike cones of $L$ coincide with those of the Lorentzian metric $A$. 
	\end{proposition}
	\begin{proof}
	To justify the above statement, we note that, in this case, $L$ and $A$ always have the same sign, which means that the future pointing timelike cones of $L$ and $A$ must be the same. Moreover,
	\begin{equation*}
		\Psi (s)=e^{P(s)};
	\end{equation*}%
	therefore, the first relation (\ref{expo_conds}) is equivalent to $\underset{%
		\dot{x}\rightarrow \partial T_{x}}{\lim }\left( A\Psi \right) =0,$ the
	second one is equivalent to:\ $\Psi -s\Psi ^{\prime }>0,$ whereas a brief
	calculation using $\boldsymbol{\sigma }=e^{P}(1-sP^{\prime })^{2}$ reveals that
	the latter one is just $\rho \dfrac{\boldsymbol\sigma ^{\prime }}{\boldsymbol\sigma }>-1.$
	\end{proof}

\vspace{5pt}
\textbf{Example:\ }A concrete example, which resembles the so-called \textit{Maxwell-Boltzmann distribution function} in the kinetic theory of gases, is obtained by considering an arbitrary metric $a$ on $M,$ together with a timelike 1-form $b\in \Omega _{1}(M)$ and:%
\begin{equation*}
		\Psi :[\left\langle b,b\right\rangle ,\infty )\rightarrow \mathbb{R}_{+},\ \
		\ \ \ \Psi (s)=e^{k-\tfrac{\left\langle b,b\right\rangle^2 }{2s^{2}}},
	\end{equation*}%
	where$\ k>0$ is a constant. The first two conditions (\ref{expo_conds}) are
	immediately checked. Besides, a brief calculation shows that the third one
	is equivalent to:%
	\begin{equation*}
		 s^{4}+\left\langle b,b\right\rangle s^{3}+5\left\langle
		b,b\right\rangle ^{2}s^{2}-\left\langle b,b\right\rangle ^{4} >0;
	\end{equation*}%
	the latter is trivially satisfied at $s=\left\langle b,b\right\rangle >0$ and the left hand 
	side is a monotonically increasing function on $(0,\infty)$, hence the inequality holds for all $s\in
	\lbrack \left\langle b,b\right\rangle ,\infty ).$

\bigskip
\section{Isometries of $(\protect\alpha ,\protect\beta )$-metric spacetimes}\label{sec:isom}

An \textit{isometry }of a pseudo-Finsler space is, \cite{Mo2}, a
diffeomorphism $\varphi :M\rightarrow M$ whose natural lift $T\varphi
:TM\rightarrow TM$ leaves $L$ invariant. Leaving aside possible discrete
symmetries, we will turn our attention to 1-parameter Lie groups of
isometries; their generators, known as \textit{Finslerian Killing vector
	fields} $\xi =\xi ^{i}\partial _{i}\in \mathcal{X}(M)$, are given by the
equation:%
\begin{equation}
\xi ^{C}(L)=0,  \label{Killing_eq_L}
\end{equation}%
where $\xi ^{C}=\xi ^{i}\partial _{i}+\xi _{,j}^{i}\dot{x}^{j}\dot{\partial}%
_{i}$ is the complete lift of $\xi $ to $TM.$

In particular, for $\left( \alpha ,\beta \right) $-metric functions $L=A\Psi
(s),$ $s=\dfrac{B^{2}}{A},$ the Killing vector condition reads: $\xi
^{C}(A)\Psi +A\Psi ^{\prime }\xi ^{C}(s)=0,$ equivalently:%
\begin{equation}
\xi ^{C}(A)\left( \Psi -s\Psi ^{\prime }\right) +2\Psi ^{\prime }B\xi
^{C}(B)=0,  \label{Killing_eq_AB}
\end{equation}%
where%
\begin{equation*}
\xi ^{C}(A)=\left( \mathfrak{L}_{\xi }a_{ij}\right) \dot{x}^{i}\dot{x}%
^{j},~\ \ \xi ^{C}(B)=\left( \mathfrak{L}_{\xi }b_{i}\right) \dot{x}^{i}
\end{equation*}%
are polynomial expressions of degree 2 and, respectively, 1 in $\dot{x}.$

\bigskip

\textbf{Particular case:\ trivial symmetries. }Assume $L$ is non-Riemannian,
that is, $\Psi ^{\prime }(s)\not=0$ at least on some interval $I_{0}\subset
I.$ If $\xi $ is a Killing vector field of $L,$ then%
\begin{equation*}
\xi ^{C}(A)=0~\ \Leftrightarrow \xi ^{C}(B)=0,
\end{equation*}%
which, differentiating with respect to $\dot{x},$ gives: 
\begin{equation*}
\mathfrak{L}_{\xi }a=0~\ \Leftrightarrow ~\ \mathfrak{L}_{\xi }b=0;
\end{equation*}%
In other words, if $\xi $ is a Killing vector field of $a,$ then it is a
Killing vector field of any $\left( \alpha ,\beta \right) $-metric $L=A\Psi $
with 1-form $b\in \Omega _{1}(M)$ invariant under the flow of $\xi $, regardless of the form of $\Psi.$

We will call Killing vectors of $L$ obeying any of the equivalent conditions $\mathfrak{L}_{\xi }a=0$ or $\mathfrak{L}_{\xi }b=0$, \textit{trivial symmetries of $L$.}

\bigskip

\textbf{Nontrivial symmetries.} In the following, let us explore which $\left( \alpha ,\beta \right) $-metrics admit nontrivial symmetries. We find the following result, which is valid for all pseudo-Finsler $\left( \alpha ,\beta \right) $ metrics, independently of their signature.

\begin{Theorem}\label{thm:isomNR}
	A non-Riemannian $\left( \alpha ,\beta \right) $ pseudo-Finsler function $%
	L=A\Psi $ admits nontrivial Killing vector fields if and only if:%
	\begin{equation}
	\Psi =\left\lbrace 
		\begin{array}{ll}
		c s^{\frac{\mu_1}{\mu_1-2\lambda_2}} \vert \mu_2 s+\mu_1-2\lambda_2 \vert^{-		\frac{2\lambda_2}{\mu_1-2\lambda_2}}, & \text{if}\ \mu_1\neq 2\lambda_2\\[5pt]
		cse^{-\frac{2\lambda_2}{\mu_2 s}},&\text{if} \ \mu_1=2\lambda_2,\, \mu_2\neq 0.
		\end{array}	 \right.
	\label{solution_Psi}
	\end{equation}%
	for some smooth functions $c,\lambda _{2},\mu _{1},\mu _{2}$ of $x$
	only, and there exist solutions $\xi \in \mathcal{X}(M)$ of the equations
	\begin{equation}
	\xi ^{C}(B)=\kappa \lambda _{2}B,\quad \xi ^{C}(A)=\kappa \left( \mu _{1}A+\mu
	_{2}B^{2}\right) ,  \label{Lie_derivs_AB}
	\end{equation}
	for some $\kappa=\kappa(x)$ which does not identically vanish. 
	If this is the case, the nontrivial Killing vectors of $L$ are precisely the solutions of \eqref{Lie_derivs_AB}.
\end{Theorem}

\begin{proof}
	$\rightarrow :$ Assume that $L$ admits a nontrivial Killing vector field $%
	\xi \in \mathcal{X}(M),$ that is, $\xi ^{C}(L)=0,$ but $\xi ^{C}(A)$ does
	not identically vanish.
	
	Fix $x\in M$ and let us restrict our attention to an (open, conic) subset
	of $\mathcal{T}_{x}$ where $\xi ^{C}(A)\not=0$; on these, we can rewrite the
	Killing equation (\ref{Killing_eq_AB}) as:%

\begin{equation}
	\dfrac{B\xi ^{C}(B)}{\xi ^{C}(A)}=\dfrac{s\Psi ^{\prime }-\Psi }{2\Psi
		^{\prime }}=\dfrac{1}{2}\left( s-\dfrac{\Psi }{\Psi ^{\prime }}\right) .
	\label{eq:Lie_derivs_AB_Psi}
\end{equation}%

	In particular, the latter expression must be equal to a ratio of two homogeneous polynomial expressions of degree two in $\dot{x}.$ But, on the other hand, it is a
	function of $B^{2}$ and $A;$ the only such possibility is:%
	\begin{equation}
	\dfrac{s\Psi ^{\prime }-\Psi }{2\Psi ^{\prime }}=\dfrac{\lambda
		_{1}A+\lambda _{2}B^{2}}{\mu _{1}A+\mu _{2}B^{2}},  \label{eq:Psi_AB}
	\end{equation}%
	for some $\lambda _{1},\lambda _{2},\mu _{1},\mu _{2}$ depending on $x$
	only. Depending on whether $\lambda _{1},\mu _{1}$ vanish or not, we
	distinguish four situations:
	
	\begin{enumerate}
		\item $\lambda _{1},\mu _{1}=0.$ In this case, $\dfrac{s\Psi ^{\prime }-\Psi 
		}{2\Psi ^{\prime }}=\dfrac{\lambda _{2}}{\mu _{2}}=:\kappa (x)$ only.
		Integration of this equation gives:%
		\begin{equation*}
		\Psi =c\left( s-2\kappa \right) ,~\ \ \ c=c(x),
		\end{equation*}%
		which means that $L=A\Psi =c\left( B^{2}-2\kappa A\right) $ is, actually,
		pseudo-Riemannian.
		
		\vspace{10pt}
		\item $\lambda _{1}=0,\mu _{1}\not=0.$ In this situation, the first equality
		(\ref{eq:Lie_derivs_AB_Psi}) gives, after simplification by $B$: 
		\begin{equation*}
		\dfrac{\xi ^{C}(B)}{\xi ^{C}(A)}=\dfrac{\lambda _{2}B}{\mu _{1}A+\mu
			_{2}B^{2}}.
		\end{equation*}
		We note that, in the right hand side, on the one hand, we must have $\lambda_{2}\neq 0$ (otherwise we only get trivial symmetries $\xi^{C}(B)=0$) and, on the other hand,  the numerator and the denominator must
		be given by relatively prime polynomials; in the contrary case, $A$ would
		admit $B$ as a factor- which is not possible, since its $\dot{x}$-Hessian 
		$\left( 2a_{ij}\right) $ is nondegenerate. But then, $B$ must divide $\xi
		^{C}(B)$, that is,%
		\begin{equation*}
		\xi ^{C}(B)=\kappa \lambda _{2}B,
		\end{equation*}%
		where, given that $\deg B=\deg \xi ^{C}(B)=1,$ we must have $\kappa =\kappa
		\left( x\right) $ only; accordingly:%
		\begin{equation*}
		\xi ^{C}(A)=\kappa \left( \mu _{1}A+\mu _{2}B^{2}\right) .
		\end{equation*}%
		The functions $\Psi $ corresponding to such symmetries are then obtained from (\ref{eq:Psi_AB}), which in our case reads
		\begin{equation}
		s-\frac{\Psi}{\Psi^{\prime}}=\frac{2\lambda_{2}s}{\mu_{1}+\mu_{2}s} \label{eq_sym_aux}
		\end{equation}
		and can be directly integrated to give
		\begin{equation}
		\Psi =c \exp \left( \int \dfrac{\left( \mu _{1}+s\mu _{2}\right) ds}{%
		\mu _{2}s^{2}+\left( \mu _{1}-2\lambda _{2}\right) s}\right).
		\end{equation}
		Calculation of the integral then leads to \eqref{solution_Psi}. We note that the situation $\mu _{1}-2\lambda _{2}=0,\,\mu_{2}=0$ is not possible, as it would lead in \eqref{eq_sym_aux}, to $\frac{\Psi}{\Psi^{\prime}}=0.$
		
		\vspace{10pt}
		\item $\lambda _{1}\not=0,\mu _{1}=0.$ This gives: $\dfrac{B\xi ^{C}(B)}{\xi
			^{C}(A)}=\dfrac{\lambda _{1}A+\lambda _{2}B^{2}}{\mu _{2}B^{2}},$
		equivalently:%
		\begin{equation}
		B^{3}\mu _{2}\xi ^{C}(B)=\xi ^{C}(A)\left( \lambda _{1}A+\lambda
		_{2}B^{2}\right) .
		\end{equation}%
		This implies that $B$ must divide $\lambda _{1}A+\lambda _{2}B^{2}$, since 
		$\xi ^{C}(A)$ is of degree 2 hence it cannot "swallow" more than a factor of $B^{2}$ of the $B^{3}$ from the left hand side. But this, in turn, implies that $B$ divides $A,$ which
		would lead to a degenerate $\dot{x}$-Hessian for $A$, which is impossible.
		
		\vspace{10pt}
		\item $\lambda _{1},\mu _{1}\not=0.$ In this case, we have:%
		\begin{equation}
		\dfrac{B\xi ^{C}(B)}{\xi ^{C}(A)}=\dfrac{\lambda _{1}A+\lambda _{2}B^{2}}{%
			\mu _{1}A+\mu _{2}B^{2}}.  \label{lambda_mu_rel}
		\end{equation}%
		The ratio in the right hand side is either irreducible, or a function of $x$
		only. This is seen as follows. Assuming that it can be simplified by a first
		degree factor, then both $\lambda _{1}A+\lambda _{2}B^{2}$ and $\mu
		_{1}A+\mu _{2}B^{2}$ must be decomposable - in particular, they must have
		degenerate $\dot{x}$-Hessians, that is:%
		\begin{equation*}
		\det \left( \lambda _{1}a_{ij}+\lambda _{2}b_{i}b_{j}\right) =\det \left(
		\mu _{1}a_{ij}+\mu _{2}b_{i}b_{j}\right) =0.
		\end{equation*}%
		Using Lemma \ref{Lemma_BCS} (see Appendix), this is: $\lambda_{1}^{4}\det a\left( 1+\dfrac{%
			\lambda _{2}	}{\lambda _{1}}\left\langle b,b\right\rangle \right) =\mu_{1}^{4}\det a\left(%
		1+\dfrac{	\mu _{2}}{\mu _{1}}\left\langle b,b\right\rangle \right) =0,$ which 
		entails:%
		\begin{equation*}
		\dfrac{\lambda _{2}}{\lambda _{1}}=\dfrac{\mu _{2}}{\mu _{1}}.
		\end{equation*}%
		But the latter actually means that the ratio $\dfrac{\lambda _{1}A+\lambda
			_{2}B^{2}}{\mu _{1}A+\mu _{2}B^{2}}=\dfrac{\lambda _{1}}{\mu _{1}}$ depends
		on $x$ only (i.e., it is actually simplified by a second degree factor, not a first degree one, as assumed).
		
		Thus, we only have two possibilities:
		
		\begin{enumerate}
			\item $\dfrac{\lambda _{1}A+\lambda _{2}B^{2}}{\mu _{1}A+\mu _{2}B^{2}}%
			=\kappa (x)$,\, that is,\, $\dfrac{s\Psi ^{\prime }-\Psi }{2\Psi ^{\prime }}%
			=\kappa $ - which, by a similar reasoning to Case 1, entails that $L$ is,
			actually, pseudo-Riemannian.
			
			\vspace{10pt}
			\item $\dfrac{\lambda _{1}A+\lambda _{2}B^{2}}{\mu _{1}A+\mu _{2}B^{2}}$ -
			irreducible. In this case, from (\ref{lambda_mu_rel}), we find that $B$ must
			either divide $\xi ^{C}(A)$ (but, then, $\dfrac{\lambda _{1}A+\lambda
				_{2}B^{2}}{\mu _{1}A+\mu _{2}B^{2}}=\dfrac{B\xi ^{C}(B)}{\xi ^{C}(A)}$ would
			equal a ratio of first degree polynomials, which contradicts the
			irreducibility assumption), or it must divide $\lambda _{1}A+\lambda _{2}B^{2}$ - which,
			taking into account that $A$ and $B$ are always relatively prime, leads to $\lambda
			_{1}=0,$ in contradiction with the hypothesis $\lambda _{1}\not=0.$
			
			Therefore, there are no properly Finslerian functions $L$ with $\lambda
			_{1},\mu _{1}\not=0.$
		\end{enumerate}
	\end{enumerate}
	
	We conclude that the only valid possibility is Case 2, thus giving the
	statement of the theorem.
	
	\bigskip
	\noindent$\leftarrow :$ Conversely, assuming that $\xi $ and $\Psi $ satisfy (\ref%
	{Lie_derivs_AB}), (\ref{solution_Psi}), a direct computation shows that $\xi
	^{C}(L)=\xi ^{C}(A)\Psi +A\Psi ^{\prime }\xi ^{C}(s)=0,$ hence $\xi $ is a Killing 
	vector for $L=A\Psi .$\\
If $L$ is non-Riemannian, these symmetries are always nontrivial. This is seen as:  $\xi^{C}(A)=0$ is only possible when $\mu_1=\mu_2=0$ -- which is, $\Psi=c$, corresponding to the Lorentzian metric $L=cA$.
\end{proof}

\bigskip

The above result gives necessary and sufficient conditions for an $\left(
\alpha ,\beta \right) $-metric to admit nontrivial symmetries. There
remains, of course, the question whether such metrics really exist. To show
that the answer is affirmative, we present below a concrete example.

\vspace{10pt}
\noindent\textbf{Example:\ nontrivial symmetries of a Bogoslovsky-Kropina spacetime
	metric. }Consider, on $M=\mathbb{R}^{4},$ a conformal deformation $a$ of the
Minkowski metric $\eta =diag(1,-1,-1,-1)$ and a lightlike 1-form $b$, as
follows:%
\begin{equation*}
a=e^{2qx^0}\eta ,~\ \ b=e^{\left( q-1\right) x^0}\left( dx^{0}+dx^{1}\right) ,
\end{equation*}%
where $q\in \left( 0,1\right) $ is arbitrarily fixed. With these data, the
Bogoslovsky-type metric:%
\begin{equation*}
L=As^{q},~\ \ \ s=\dfrac{B^{2}}{A}=\dfrac{(b\left( \dot{x}\right) )^{2}}{%
	a\left( \dot{x},\dot{x}\right) },
\end{equation*}
is a nontrivial Finsler spacetime function, whose $\Psi (s)=s^{q},$ fits
into the class (\ref{solution_Psi}) (for $\lambda _{2}=q-1,~\ \ \mu
_{1}=2q,~\ \ \mu _{2}=0,~\ \ \kappa =c=1$).

We take as $\xi ,$ the time translation generator $\xi =\partial _{0},$
which gives:%
\begin{equation*}
\xi ^{C}=\partial _{0}.
\end{equation*}%
Then, using $A=e^{2qx^0}\eta _{ij}\dot{x}^{i}\dot{x}^{j},$ $B=e^{\left(
	q-1\right) x^0}\left( \dot{x}^{0}+\dot{x}^{1}\right) ,$ we immediately find:%
\begin{equation*}
\xi ^{C}(A)=2qA\not=0,~\ \ \xi ^{C}(B)=\left( q-1\right) B
\end{equation*}%
which are precisely the conditions \eqref{Lie_derivs_AB}. Therefore, $\xi $ is a nontrivial Killing
vector field for $L$.

\bigskip
\section{Conclusion}\label{sec:conc}
Among all possible Finsler metrics, the class of $(\alpha,\beta)$-metrics is of particular interest since it allows for numerous applications and explicit calculations. So far, it was not known how to identify $(\alpha,\beta)$-Finsler spacetimes in general. These serve as candidates for the application in physics, in particular in gravitational physics.

In this article, we identified the necessary and sufficient conditions such that an $(\alpha,\beta)$-Finsler metric defines a Finsler spacetime in Theorem \ref{thm:sign_g};  these are the conditions that ensure Lorentzian signature inside a convex cone with null boundary, in each tangent space. As we demonstrated afterwards, this enables us to find the best candidates for physical applications belonging to different subclasses of $(\alpha,\beta)$-metrics. For Randers metrics, we found that the defining $1$-form must be non-spacelike and with bounded norm w.r.t. to the  pseudo-Riemannian metric (Proposition \ref{prop:rand}); for Bogoslovky-Kropina metrics, the value of the power parameter is constrained (Proposition \ref{prop:bog}); for Kundt and exponential type metrics, we provided simple necessary and sufficient Finsler spacetime constraints in Propositions \ref{prop:kundt} and, respectively, \ref{prop:exp}.

Moreover, we investigated the existence of isometries of $(\alpha,\beta)$-metrics. Surprisingly, it turned out that there can exist isometries of the Finsler metric which are not isometries of the metric\footnote{A contrary statement appears in the arxiv version of the paper \cite{Mo1}; yet, in the published version \cite{Mo2}, this statement does not appear anymore.} 
$a$, as we pointed out in Theorem \ref{thm:isomNR} and subsequently, in a concrete example. The physical consequences of the existence of these isometries still has to be worked out. A future project is, e.g., to construct the most general spherical symmetric $(\alpha,\beta)$-metric, whose spacetime metric $a$ is not spherically symmetric. 

\bigskip
\authorcontributions{The authors have all contributed substantially to the derivation of the presented results as well as analysis, drafting, review, and finalization of the manuscript. All authors have read and agreed to the published version of the manuscript.}

\funding{ADD}
 CP is funded by the Deutsche Forschungsgemeinschaft (DFG, German Research Foundation) - Project Number 420243324 and acknowledges the excellence cluster Quantum Frontiers funded by the Deutsche Forschungsgemeinschaft (DFG, German Research Foundation) under Germany’s Excellence Strategy – EXC-2123 QuantumFrontiers – 390837967.

\acknowledgments{ADD}
This article is based upon work from COST Action CA21136 (Addressing observational tensions in cosmology with systematics and fundamental physics - CosmoVerse) and the authors would like to acknowledge networking support by the COST Action CA18108 (Quantum Gravity Phenomenology in the Multi-Messenger Approach), supported by COST (European Cooperation in Science and Technology).
The authors are grateful to Andrea Fuster and Sjors Heefer for numerous useful talks and insights.


\conflictsofinterest{The authors declare no conflict of interest.}


\appendix

\bigskip
\section{Calculation of $\det (g_{\dot{x}})$}\label{app:detg}

Consider, in the following, $\dim M=n$ and a pseudo-Finsler $\left( \alpha
,\beta \right) $-metric $L=A\Psi :\mathcal{A}\rightarrow \mathbb{R}.$ We
will prove in the following that: 
\begin{equation}
	\det \left( g_{\dot{x}}\right) = \Psi ^{2}\left( \Psi -s\Psi^{\prime }\right) ^{n-3}\det(a)\dfrac{\partial }{\partial s}%
	\left( \left( s-\left\langle b,b\right\rangle \right) \dfrac{\left( \Psi
		-s\Psi ^{\prime }\right) ^{2}}{\Psi }\right)
  \label{det_g}
\end{equation}%
where primes denote differentiation with respect to $s$ and $\left\langle
b,b\right\rangle :=a^{ij}b_{i}b_{j}.$

To this aim, the first step is to see that, in any local chart:%
\begin{equation}
	g_{ij}=\left( \Psi -s\Psi ^{\prime }\right) a_{ij}+\Psi ^{\prime }b_{i}b_{j}+%
	\dfrac{1}{2}A\Psi ^{\prime \prime }s_{\cdot i}s_{\cdot j}  \label{g_ij_Psi}
\end{equation}%
this is easily obtained by differentiating $L$ with respect to $\dot{x}^{i},$
$\dot{x}^{j}$ and using the relations $A_{\cdot ij}=2a_{ij}$ and:%
\begin{equation*}
	As_{\cdot ij}+s_{\cdot i}A_{\cdot j}+s_{\cdot j}A_{\cdot
		i}=2b_{i}b_{j}-2sa_{ij}\ \ ;
\end{equation*}%
the latter follows by differentiation of the identity $B^{2}=sA$ twice, with 
$\dot{x}^{i}$ and $\dot{x}^{j}$.

\bigskip

The above formula for $g_{ij}$ suggests that, in order to calculate $\det
(g_{ij})$ we can apply twice the following known result from linear algebra:

\begin{lemma}
	\label{Lemma_BCS} (\cite{CS}, pg.4): \textit{If }%
	$Q\in \mathcal{M}_{n}\left( \mathbb{C}\right) $ has\textit{\ }$\det
	(Q_{ij})\not=0,$ then:
	
	\begin{enumerate}
		\item $\det (Q_{ij}+\delta C_{i}C_{j})=\det (Q_{ij})\left( 1+\delta C^{k}C_{k}%
		\right) ,$ for all $n$-vectors $C=(C_{i})\in \mathbb{C}^{n}$, $\delta\in\mathbb R,$ 
		where the indices of $C_{i}$ have been raised by means of the inverse 
		$\left( Q^{ij}\right) .$
		
		\vspace{5pt}
		\item If $1+\delta C^{k}C_{k}\not=0,$ the inverse of the matrix $\left(
		Q_{ij}+\delta C_{i}C_{j}\right) $ exists and has the entries:%
		\begin{equation}
			Q^{ij}-\dfrac{\delta C^{i}C^{j}}{1+\delta C^{k}C_{k}}.  \label{Q_inv}
		\end{equation}%
		
	\end{enumerate}
\end{lemma}

Indeed, applying twice the above Lemma, we easily find:

\begin{corollary}
	\textit{If }$Q\in \mathcal{M}_{n}\left( \mathbb{R}\right) $ has\textit{\ }$%
	\det (Q_{ij})\not=0,$ $\delta ,\mu \in \mathbb{R}$ and the components of the
	n-vectors $B_{k},C_{k}\in \mathbb{R}$ are such that at least one of the
	quantities $1+\delta B^{k}B_{k},$ $1+\mu C^{k}C_{k}$ is nonzero, then:%
	\begin{equation}
		\det \left( Q_{ij}+\delta B_{i}B_{j}+\mu C_{i}C_{j}\right) =\det \left(
		Q_{ij}\right) \left[ \left( 1+\delta B^{k}B_{k}\right) \left( 1+\mu
		C^{k}C_{k}\right) -\delta \mu \left( B^{k}C_{k}\right) ^{2}\right] ,
		\label{double_lemma}
	\end{equation}%
	where $B^{k}=Q^{kj}B_{j},$ $C^{k}=Q^{kj}C_{j}.$
	
	Moreover, if $ \left( 1+\delta B^{k}B_{k}\right) \left( 1+\mu C^{k}C_{k}\right) -
	\delta \mu \left( B^{k}C_{k}\right) ^{2}\neq0$,  the inverse of the matrix $
	(Q_{ij}+\delta B_{i}B_{j}+\mu C_{i}C_{j})$ exists and has the entries:
	\begin{equation}
		\label{double inverse}
		Q^{ij}-\dfrac{\mu \left( 1+\delta B^kB_k \right)C^iC^j-\mu\delta B^k C_k \left( B^i C^j + 
		B^jC^i \right)+\delta\left(1+\mu C^kC_k \right)B^iB^j}{\left( 1+\delta B^{k}B_{k}
		\right) \left( 1+\mu C^{k}C_{k}\right) - \delta \mu \left( B^{k}C_{k}\right) ^{2}}
	\end{equation}
\end{corollary}

\begin{proof}
	Denote $\hat{Q}_{ij}=Q_{ij}+\delta B_{i}B_{j}.$ Assuming that $1+\delta
	B^{k}B_{k}\not=0,$ from Lemma \ref{Lemma_BCS}, we first find:%
	\begin{equation*}
		\det (\hat{Q}_{ij})=\det (Q_{ij})\left( 1+\delta B^{k}B_{k}\right) \not=0,
	\end{equation*}%
	therefore $(\hat{Q}_{ij})$ is invertible, with inverse as in (\ref{Q_inv}).
	The result then follows by applying again Lemma \ref{Lemma_BCS}, for the
	matrix with entries $\hat{Q}_{ij}+\mu C_{i}C_{j}.$ The case when $1+\mu
	C^{k}C_{k}\not=0$ is completely similar.
	
The formula for the inverse matrix can be obtained similarly, by applying twice point 2) of the above Lemma.
\end{proof}

\bigskip

We are now ready to calculate $\det (g_{\dot{x}}).$ To this aim, let us fix
an arbitrary $x\in M$ and an arbitrary local chart on $\pi ^{-1}(x)\subset
TM.$ Taking into account Lemma \ref{lem:sign_Psi}, there is no loss of
generality if we limit our attention to conic subsets of $TM$ where $\Psi
-s\Psi ^{\prime }\not=0.$ On such subsets, we can write:
\begin{equation}
    \label{g_ij_delta_mu}
    	g_{ij}=(\Psi -s\Psi ^{\prime})\left( a_{ij}+\delta b_{i}b_{j}+\mu s_{\cdot i}s_{\cdot j}\right),
\end{equation}
hence we can apply the above Corollary, for:%
\begin{equation}
	Q_{ij}=a_{ij}\,,\quad  B_{k}:=b_{k}\,,\quad C_{k}=s_{\cdot k}\,,\quad \delta =\dfrac{\Psi
		^{\prime }}{\Psi -s\Psi ^{\prime }}\,,\quad \mu =\dfrac{1}{2}\dfrac{A\Psi
		^{\prime \prime }}{\Psi -s\Psi ^{\prime }}.  \label{def_delta_mu}
\end{equation}

\begin{itemize}
	\item To calculate the blocks appearing in (\ref{double_lemma}), we use $%
	A_{\cdot i}=2a_{ik}\dot{x}^{k}$, which yields: $a^{ij}A_{\cdot i}=2\dot{x}%
	^{j},a^{ij}A_{\cdot i}b_{j}=2B,s_{\cdot i}=\frac{2}{A}\left( Bb_{i}-s%
	\widetilde{\dot{x}}_{i}\right) $ (where $\widetilde{\dot{x}}_{i}=a_{ij}\dot{x%
	}^{j}$) and finally: 
	\begin{equation}
	\label{def_blocks}
		B^{k}B_{k}=\left\langle b,b\right\rangle ,\qquad C^{k}C_{k}=\dfrac{4s}{A}%
		\left( \left\langle b,b\right\rangle -s\right) ,\qquad \left(
		B^{k}C_{k}\right) ^{2}=\dfrac{4s}{A}\left( \left\langle b,b\right\rangle
		-s\right) ^{2}.
	\end{equation}
	
	\item On subsets where $1+\delta b^{k}b_{k} \neq 0$,
	we can apply the Corollary, which gives, after a brief computation:%
	\begin{equation*}
		\det \left( g_{ij}\right) =(\Psi -s\Psi ^{\prime})^{n-2}\det(a_{ij})\left[ \Psi \left( \Psi -s\Psi ^{\prime
		}\right) +\left( \left\langle b,b\right\rangle -s\right) \left( \Psi \Psi
		^{\prime }+2s\Psi \Psi ^{\prime \prime }-s{\Psi ^{\prime }}^{2}\right)\right] 
	\end{equation*}%
	The square bracket can be rewritten as: 
	\begin{equation*}
		\dfrac{\Psi ^{2}}{\left( \Psi -s\Psi ^{\prime }\right) }\dfrac{\partial }{%
			\partial s}\left( \left( s-\left\langle b,b\right\rangle \right) \dfrac{%
			\left( \Psi -s\Psi ^{\prime }\right) ^{2}}{\Psi }\right),
	\end{equation*}
	which gives the desired result (\ref{det_g}).
	
	\vspace{5pt}
	\item The result can then be prolonged by continuity also at points where $1+\delta b^{k}b_{k} = 0$. To see this, we note that this equality cannot happen on any entire interval $I_{0}\subset I$, as this would entail, that, for $s \in I_{0}$:
	\begin{equation*}
	1+\dfrac{\Psi ^{\prime }\left\langle b,b\right\rangle }{\Psi -s\Psi
		^{\prime }}=0~\qquad \Leftrightarrow \qquad \dfrac{\Psi ^{\prime }}{\Psi }=%
	\dfrac{1}{s-\left\langle b,b\right\rangle}.
	\end{equation*}%
	that is, $\Psi
	=\kappa \left( s-\left\langle b,b\right\rangle \right) ,$ where $\kappa
	=\kappa (x)$ only; but then, on the (open) subset $s^{-1}(I_{0})\subset 
	\mathcal{A}$ we would have: $L=A\Psi =\kappa \left( B^{2}-\left\langle
	b,b\right\rangle A\right) ,$ which has degenerate Hessian $g_{ij}=\kappa
	\left( b_{i}b_{j}-\left\langle b,b\right\rangle a_{ij}\right) $ - and hence
	does not represent a pseudo-Finsler function.

\end{itemize}

In particular, for $n=\dim M=4,$ we find: 
\begin{equation}
\det \left( g_{\dot{x}}\right) = \Psi ^{2}\left( \Psi -s\Psi^{\prime }\right)\det(a)\dfrac{\partial }{\partial s}%
\left( \left( s-\left\langle b,b\right\rangle \right) \dfrac{\left( \Psi
	-s\Psi ^{\prime }\right) ^{2}}{\Psi }\right)
\label{det_g_4}
\end{equation}

Now, if the determinant is nonzero, i.e. $\left( s-
\left\langle b,b\right\rangle \right) \dfrac{\left( \Psi -s\Psi^{\prime }\right) ^{2}}
{\Psi }=\ const.$,  we can apply again the result (\ref{double inverse}) from the above Corollary for the metric tensor (\ref{g_ij_delta_mu}). Using the blocks (\ref{def_blocks}), after some calculations we obtain the inverse 
\begin{align}
	\label{g^ij_Phi}
	g^{ij}=\frac 1 {\Psi-s \Psi '}a^{ij}& + \frac{2s^2}{ \nu A} \Psi''[\Psi'(s- \left\langle b,b\right\rangle )-\Psi]\dot x^i \dot x^j\\
	&+\frac{2Bs}{\nu A} \Psi\Psi''(b^i\dot x^j+b^j\dot x^i) - \frac{1}{\nu} [\Psi'(\Psi-s\Psi')+2s\Psi\Psi'']b^i b^j, \notag
\end{align}
where 
\begin{equation}
	\notag
	\nu=(\Psi-s\Psi')\left[ \Psi \left( \Psi -s\Psi ^{\prime
		}\right) +\left( \left\langle b,b\right\rangle -s\right) \left( \Psi \Psi
		^{\prime }+2s\Psi \Psi ^{\prime \prime }-s{\Psi ^{\prime }}^{2}\right)\right]\,,
\end{equation}
The  same results can be also  obtained using   \cite{Fuster-VGR}.

 Notice that $\nu$ defined above can be related to $\sigma$ and $\rho$ from Theorem \ref{thm:sign_g} as
\[ \nu=\Psi^2 \dfrac{d}{ds}(\rho\sigma).\]

\reftitle{References}



\externalbibliography{yes}
\bibliography{AB}

\end{document}